\documentclass[11pt,a4paper]{article}

\usepackage{amsfonts,amsgen,amsmath,amstext,amsbsy,amsopn,amsthm,amssymb
}
\usepackage{color}
\usepackage{multirow}
\usepackage[pdfencoding=auto, psdextra]{hyperref}
\usepackage[noabbrev,capitalise]{cleveref}
\usepackage[title]{appendix}

\newtheorem{dfn}{Definition} [section]

\newtheorem{thm}[dfn]{Theorem}
\newtheorem{lem}[dfn]{Lemma}

\newtheorem{claim}[dfn]{Claim}

\renewcommand{\d}{{\rho}}
\renewcommand{\Im}{\operatorname{Im}}

\newcommand{\SK}{{S_{k,\ell}}}

\newcommand{\eps}{\varepsilon}
\newcommand{\E}{\mathbb{E}}

\newcommand{\fs}[2]{\left(\frac{#1}{#2}\right)}
\newcommand{\s}[1]{\left(#1\right)}
\newcommand{\mc}[1]{\mathcal{#1}}

\newcommand{\bb}[1]{\mathbb{#1}}

\newcommand{\brm}[1]{\operatorname{#1}}

\newcommand{\x}{\alpha} % |X| size
\newcommand{\y}{\beta} % |Y| size
\renewcommand{\Im}{\mathrm{Im}}
\renewcommand{\l}{\ell}

\begin{document}

\title{\bf\Large The inducibility of oriented stars}

\author{Ping Hu\thanks{School of Mathematics, Sun Yat-sen University, Guangzhou, 510275, China. Email: {\tt huping9@mail.sysu.edu.cn.} Supported in part by National Natural Science Foundation of China grants 11801593 and 11931002.}
\and
Jie Ma\thanks{School of Mathematical Sciences, University of Science and Technology of China, Hefei 230026, China. Email:{\tt
	jiema@ustc.edu.cn.} Supported in part by National Natural Science Foundation of China grant 11622110 and Anhui Initiative in Quantum information Technologies grant AHY150200.}
\and
Sergey Norin\thanks{Department of Mathematics and Statistics, McGill University, Montr\'eal, Canada. Email: {\tt sergey.norin@mcgill.ca}. Supported by an NSERC Discovery grant.}
\and
Hehui Wu\thanks{Shanghai Center for Mathematical Sciences, Shanghai, China. E-mail:{\tt hhwu@fudan.edu.cn}. Supported in part by National Natural Science Foundation of China grant 11931006 and the Shanghai Shuguang Scholar Program grant 19SG01.}}

\maketitle

\begin{abstract}
	We consider the problem of  maximizing the number of induced copies of an oriented star $S_{k,\l}$ in digraphs of given size,
	where the center of the star has out-degree $k$ and in-degree $\ell$.
	The case $k\l=0$ was solved by Huang in \cite{Huang14}.
	Here, we asymptotically solve it for all other oriented stars with at least seven vertices. %In particular, we show that extremal digraphs are complete bipartite.
\end{abstract}

%\medskip
%\noindent {\bf Keywords:} Inducibility, Induced density, Oriented star

%\medskip
%\noindent {\bf Mathematics Subject Classification (2010):}

\section{Introduction}

A central problem in extremal graph theory is to determine the maximum number of induced copies of any given graph $H$ in graphs with fixed size.
This problem was first studied by Pippenger and Golumbic \cite{PG} and has been the subject of extensive research in recent years \cite{BHLP,BS,HHNBlowup,HT,KNV,Yuster19}.
%Despite of many exciting work, the answers for many basic graphs such as cycles remain unknown.

%Given a digraph $H$, \emph{the induced density} $i(H,G)$ is the probability that a random map $\phi: V(H) \to V(G)$ is an isomorphism between $H$ and $G[\phi(V(H)]$. For integers $n$ let $i(H,n)$ be the maximum of $i(H,G)$ over all $n$-vertex
%digraphs $G$, and let the {\it inducibility} of $H$, denoted $i(H)$, be $\lim\limits_{n\to \infty} i(H,n)$. Let us notice that the inducibility of any $H$ exists.

Our focus in this paper is the analogous problem for digraphs.
To be precise, let $H$ be a digraph.
The {\it induced density} of $H$ in a digraph $G$, denoted by $i(H,G)$, is the number of induced copies of $H$ in $G$ divided by $\binom{|V(G)|}{|V(H)|}$.
For integers $n$, let $i(H,n)$ be the maximum of $i(H,G)$ over all $n$-vertex digraphs $G$.
The {\it inducibility} of $H$ is defined to be $i(H)=\lim_{n\to \infty} i(H,n)$. 
This limit exists as $i(H,n)$ is decreasing for $n \geq 2$.

There are very few digraphs for which the inducibility is known. One important class of  examples are directed stars. For nonnegative integers $k$ and $\ell$, let the \emph{oriented star} $S_{k,\ell}$ be  the digraph obtained by directing edges of a star with $k+\ell$ leaves so that center has out-degree $k$ and in-degree $\ell$. A \emph{directed star} is an oriented star in which all the edges have the same direction, i.e. the star $S_{k,\ell}$ such that $k=0$ or $\ell=0$.  The inducibility of $S_{2,0}$ and $S_{3,0}$ was determined by Falgas-Ravry and Vaughan~\cite{FV}. 
Resolving a conjecture made in \cite{FV}, Huang \cite{Huang14} extended their result and determined the inducibility of $S_{k,0}$ for all $k\geq 2$, showing that it is asymptotically attained by an unbalanced blow-up of an arc, 
iterated inside the part with in-degree $0$. Note that since the inducibility of any digraph equals the inducibility of the digraph obtained by reversing all arcs, it suffices to consider oriented stars $S_{k,\ell}$  such that $k \geq \ell$. In particular, Huang's result also determines inducibility of $S_{0,\ell}$ for all $\ell$.

The smallest oriented star not covered by the result of \cite{Huang14} is $S_{1,1}$, the directed path on three vertices. Thomass\'e~\cite[Conjecture 6.32]{CHSurvey} conjectured that $i(S_{1,1})=2/5$, which is attained by the iterated blowup of the directed cycle on four vertices.   

In this paper we determine the inducibility of the remaining oriented stars on at least seven vertices. The following is our main result, which reduces computations $i(\SK)$ to an optimization problem.

\begin{thm}\label{Thm:main}
	Let $k\geq \ell\geq 1$ be integers such that $k+\ell \geq 6$.
	Then when $k=\ell$,
	$$i(S_{k,\ell})=\frac{(2k+1)!}{2^{2k}(k!)^2}\cdot \max_{\alpha}\left\{\alpha(1-\alpha)^{2k}+(1-\alpha)\alpha^{2k}\right\};$$
	and when $k\geq \ell+1$, $i(S_{k,\ell})$ is equal to
	$$\frac{(k+\l+1)!}{k!\ell!}\max_{\alpha, d}\left\{\alpha(1-\alpha)^{k+\l}d^k(1-d)^\ell+\frac{(k-1)^{k-1}\l^\l}{(k+\ell-1)^{k+\l-1}}(1-\alpha)\alpha^{k+\l}(1-d)\right\},$$ 
	where the maximum is over all possible pairs $(\alpha,d)\in [0,\frac12]\times [0,\frac{k}{k+\ell}]$.
\end{thm}

The lower bound on $i(S_{k,\ell})$ in \cref{Thm:main}  comes from considering  orientations of
large complete bipartite graph with bipartition $(X, Y)$ with $|X|/|Y|=\alpha/(1-\alpha)$ and $d|X||Y|$ arcs directed from $X$ to $Y$,
where  $\alpha = (1+o(1))\frac{1}{k+\ell+1}$, $d = (1+o(1))\frac{k}{k+\ell}$.\footnote{For a more detailed description of the extremal digraphs, we direct readers to \cref{l:lower}.} In fact, it follows from our proof  that for sufficiently large $n$ the graphs achieving $i(\SK, n)$ are the orientations of complete bipartite graphs with above properties. 

As discussed above, the optimal constructions for $\ell = 0$ and for $k = \ell = 1$ are different, and it is not hard to show that orientations of complete bipartite graphs can not achieve the inducibility in these cases.  However, we conjecture that \cref{Thm:main} holds in all the remaining cases. \footnote{Tightening some of our estimates, it is possible to extend \cref{Thm:main} to the cases $(k,\l)=(2,2)$ and $(k,\l)=(3,2)$, but we choose not to do so, as it involves no new ideas, and noticeably lengthens the (already extensive) calculations. %The case $(k,\l)=(2,1)$ can be numerically verified using flag algebras. 
% The only remaining cases are $k \in \{2,3,4\}$ and $\l=1$.}   
On the other hand, solving the cases $(k,1)$ for $k\leq 4$ may require some new ideas (see \cref{s:remarks} for more discussion).}

The rest of the paper is structured as follows. 
In Section \ref{section:Pre}, we outline the proof of \cref{Thm:main}, introduce the necessary terminology, including the probabilistic notation that plays a crucial role in the proofs, and establish some basic inequalities.
In Section \ref{s:lower}, we prove the stated lower bound on $i(\SK)$ and reduce the proof of the upper bound to \cref{l:upper}.  We then prove \cref{l:upper} in \cref{s:upper}, while the proofs of some of the more technical estimates are relegated to Appendix \ref{sec:claim}. 
Finally, in
\cref{s:remarks}, we make some concluding remarks.

\section{Preliminaries}\label{section:Pre}

\subsection{Proof outline}

The techniques used to obtain majority of inducibility results can be loosely divided into three categories: 
\begin{itemize}
	\item {\bf Razborov's flag algebras} allow one to automate the search for the proof. Their application, however, frequently requires computer assistance, and, as such, is typically more successful when investigating small (di)graphs. Flag algebras, in particular, were used by Hirst~\cite{Hirst} to determine the inducibility of the graphs obtained from a triangle by adding a vertex of degree one or two, and by Falgas-Ravry and Vaughan~\cite{FV} to determine the inducibility of $S_{2,0}$ and $S_{3,0}$.
 \item {\bf Stability method} was introduced by Simonovits~\cite{Sim68}. In the context of inducibility operates as follows. First, one establishes that if the induced density of a  (di)graph $H$ in a (di)graph $G$ is close to the maximum then $G$ is close to the conjectured family of extremal examples $\mc{F}$.  Second, one locally modifies $G$, arguing that the induced density increases in the process, until reaching the graph that belongs to $\mc{F}$. This technique requires fine control on the error terms. As a result it is either paired with flag algebras ~\cite{BHLP,CLP20,PST}, or is applied to the cases when $H$ is very large~\cite{FSW,HHNBlowup}.
 \item{\bf Analytical methods} are applied by establishing a bound on the inducibility as a function of local parameters (e.g.\@ degrees) and optimizing this bound after relaxing the integrality requirement. Huang's result~\cite{Huang14} on the inducibility of directed stars is an example of applications of such methods, which were also used in~\cite{BS,KNV}. Success of analytic methods seems to depend on the presence symmetries in $H$, in particular in all of the above applications  the (di)graph  $H$ is edge- (or arc-)transitive. 
\end{itemize}

One of the goals of this paper is addressing the technical challenge of bridging the gap between small digraphs, which can be attacked using flag algebras, and large digraphs, which are susceptible to stability method, and doing so in the presence of only moderate symmetries.  

\medskip

The large scale structure of our proof of \cref{Thm:main} is motivated by the stability method. More precisely, we partition the vertices of $G$ into a set $X$ of large degree vertices and a set $Y$ of small degree vertices, and attempt to show that the induced density is maximized when $X$ and $Y$ are independent and complete to each other. 
However, rather than sequentially making local modifications, we employ analytical techniques and show that global contribution of edges in $X$ and $Y$ and non-adjacencies between $X$ and $Y$ is negative, while we simultaneously optimize the degrees of vertices in $X$ and $Y$ to establish the claimed upper bound. As a result we are able to reduce the size of the stars for which this approach is feasible to the point that the remaining cases are within reach of the flag algebras.

Our approach overcomes the following issue that makes the direct application of stability method difficult, which occurs when $\l=1$. Consider a digraph $G$ on $n$ vertices, obtained by ``blowing up'' the oriented star $S_{k,1}$ in the following iterated fashion. The vertex set is partitioned into three sets $X$, $Y$ and $Z$, where $|X|=|Z|=(1+o(1))\frac{n}{k+2}$ and $|Y|=(1+o(1))\frac{k}{k+2}n$, the edge set $E(G)$ includes all the arcs from $X$ to $Y$ and from $Y$ to $Z$, and this construction is iterated within $G[X]$ and $G[Y]$. The resulting digraph is not an orientation of a complete bipartite graph, as we would like it to be, yet the induced  density of $S_{k,1}$ in it is very close to the maximum, and  it is locally optimal, i.e., it is impossible to increase induced density by changing the arcs incident to a single vertex. The existence of this example makes analysis in the case $\l=1$ substantially more complicated, although we never explicitly reference it in the course of the proof.

We use the probabilistic notation introduced in the next subsection to make our calculations more intuitive. That is, instead of counting the copies of $\SK$ in $G$, we bound the probability that a random map from $V(\SK)$ to $V(G)$ induces such a copy.

\subsection{Probabilistic point of view}

For the remainder of the paper we fix integers $k \geq \l \geq 1$ such that $k + \l \geq 6$.

Let $G$ be an $n$-vertex digraph.
We consider the set $\Phi=\Phi_{k,\l}$ of all maps $\phi: V(S_{k,\l})\to V(G)$ with the uniform probability distribution on them. 
Let $\mathcal{S}$ be the set of maps $\phi \in \Phi$ such that $\phi$ is an isomorphism between $\SK$ and $G[
\operatorname{Image}(\phi)]$,
and let $$s(G)=s_{{k,\l}}(G)=\Pr[\phi \in \mathcal{S}].$$ Then we have \begin{equation}\label{e:s} s(G)=\frac{k!\l!}{n^{k+\l+1}}\binom{n}{k+\l+1}\cdot i(S_{k,\l},G).
\end{equation}Thus  if $G$  maximizes the induced density of $\SK$ among all $n$-vertex digraphs, then it also maximizes  $s(G)$. We find it convenient to take the probabilistic approach and estimate $s(G)$ rather than $i(S_{k,\l},G)$, which also allows us to largely ignore the dependence on $n$.

Let \begin{align*}\brm{OPT}(k,\l) &=  \frac{1}{2^{2k}}\cdot \max_{\alpha}\left\{\alpha(1-\alpha)^{2k}+(1-\alpha)\alpha^{2k}\right\}  \; \mbox{if} \; k=\l, \; \mbox{and}\\ 
\brm{OPT}(k,\l) &= \max_{\alpha, d}\left\{\alpha(1-\alpha)^{k+\l}d^k(1-d)^\ell+\frac{(k-1)^{k-1}\l^\l}{(k+\ell-1)^{k+\l-1}}(1-\alpha)\alpha^{k+\l}(1-d)\right\} \; \mbox{if} \; k>\l.\end{align*}

Thus by \eqref{e:s} \cref{Thm:main} is implied by the following.

\begin{thm}\label{t:main2} Let $k \geq \ell \geq 1$ be integers such that $k+\ell \geq 6$. Then
	$$ \sup_G s_{k,\l}(G) = \brm{OPT}(k,\l).\footnote{We remark that one can show $\sup_G s_{k,\l}(G) =\limsup_G s_{k,\l}(G)$.}$$ 
\end{thm}

\subsection{Notation}

We write $xy$ for an arc in a digraph $G$ with the head $y$ and the tail $x$ (when the direction is not essential, we also call it an edge in $G$).

Let $A$ be a subset of $V(G)$. Let $\mu(A) = |A|/n$ denote the proportion of vertices in $G$ that lie in $A$.
We denote by $N^+_A(x)$ (respectively $N^-_A(x)$)  the set of all out-neighbors (respectively in-neighbors)
of $x$ in $G$ contained in $A$, and let $N_A(x)=N^+_A(x)\cup N^-_A(x)$. Let
\begin{align*} \d^+_A(x)=|N^+_A(x)|/n, ~~ \d_A(x)=|N_A(x)|/n, ~~ \d^-_A(x)=|N^-_A(x)|/n \mathrm{~~~and~~~   } \d^0_A(x)= \mu(A) - \d_A(x).\end{align*} 
Thus $\d^0_A(x)$ denotes the proportion of vertices which lie in $A$ and are non-adjacent to $x$.
In the case of $A=V(G)$ we  omit the subscript.

As we investigate maps in $\Phi$ it is convenient to fix notations for vertices of $\SK$. Let $c$ denote the center of $S_{k,\l}$, and let $I$ and $O$ be the sets of in-leaves and out-leaves of $S_{k,\l}$, respectively.
We fix particular choices of $i \in I$ and $o \in O$.

For a vertex $v\in V(G)$, and $z \in V(\SK)$, let $$s(z \to v) = \Pr\left[\phi \in \mathcal{S}|\phi(z)=v\right].\footnote{We will omit the dependence on $G$ in our notation, as the choice of $G$ will be fixed throughout the bulk of the proof.}$$ 
Given $U\subseteq V(G)$, let $$s(U)=\Pr[\phi \in \mathcal{S}|U\subseteq \textrm{Im}(\phi)].$$
Typically, we list vertices of $U$ and omit the brackets for brevity. For example, $s(v)=s(\{v\})$ and $s(u,v)=s(\{u,v\})$.

Note that by the law of total probability, we have
\begin{align}\label{eq:sgv}
s(v)=\frac{s(c\to v)+\l\cdot s(i\to v)+k\cdot s(o\to v)}{k+\l+1}
\end{align}
and
\begin{align}\label{eq:sG}
s(G)=\sum_{v\in V(G)}\frac{s(v)}{n}=\sum_{(u,v)\in V(G)}\frac{s(u,v)}{n^2} =\sum_{v\in V(G)}\frac{s(z\to v)}{n},
\end{align}
where the last equality holds for any choice of $z \in V(\SK)$.

\subsection{Basic inequalities}

In this subsection we collect the basic inequalities and analytic tools we use throughout the paper. 

\begin{lem}\label{l:xayb}
 For  reals $a, b>0, x, y \geq 0$, we have
	\begin{align}\label{e:xayb}
x^ay^b\leq \frac{a^ab^b}{(a+b)^{a+b}}(x+y)^{a+b}.
	\end{align}
\end{lem}
\begin{proof}
	Clearly, \eqref{e:xayb} holds if $xy=0$, so we assume $x,y >0$.  Taking the logarithm and dividing by $a+b$, we transform \eqref{e:xayb} into
	$$\frac{a}{a+b}\log\fs{x}{a}+\frac{b}{a+b}\log\fs{y}{b} \leq \log\fs{x+y}{a+b},$$
	which holds by concavity of logarithm.
\end{proof}

\begin{lem}\label{l:fractions}
	Let $a,b >0$, $0 \leq x \leq a$, $0 \leq y \leq b$ be reals. Then
	\begin{equation}\label{e:subtract}
	(a-x)(b-y)  \leq ab \s{1 - \frac{x+y}{a+b}}.
	\end{equation}
\end{lem}
\begin{proof}
	After multiplying by $a+b$, the inequality \eqref{e:subtract} reduces to $ay(a-x)+bx(b-y) \geq 0$, which holds by our assumptions.
\end{proof}

Let $$f(x)=\begin{cases} \frac{(k-1)^{k-1}\l^\l}{(k+\l-1)^{k+\l-1}}\cdot x, &\mbox{if} \qquad x \in \left[0,\frac{k-1}{k+\l-1}\right] \\
 x^k(1-x)^{\l}, &\mbox{if} \qquad x \in \left[\frac{k-1}{k+\l-1},\frac{k}{k+\l}\right]  \\
 \frac{k^k\l^\l}{(k+\l)^{k+\l}}, &\mbox{if} \qquad x \in \left[\frac{k}{k+\l},1\right]. \end{cases} $$

\begin{lem}\label{l:f}	The function $f$ is non-decreasing, concave and satisfies $f(x) \geq x^k(1-x)^{\l}$ on $[0,1]$.
\end{lem}

\begin{proof} 
	Let $g(x)=x^k(1-x)^{\l}$, so $g'(x) = x^{k-1}(1-x)^{\l-1}(k-(k+\l)x)$.
	% Then $g'(x)$ is decreasing for $x \in [\frac{k}{k+l} - \frac{\sqrt{kl}}{k+l}, \frac{k}{k+l} + \frac{\sqrt{kl}}{k+l}]$. 
	% As $\frac{k}{k+l} - \frac{\sqrt{kl}}{k+l} \leq \frac{k-1}{k+\l-1}$ and $\frac{k}{k+l} + \frac{\sqrt{kl}}{k+l} \geq \frac{k}{k+\l}$, $f(x)$ is concave. 
% 	As $g(x)$ is increasing on $[0,\frac{k}{k+\l}]$, $g(x)$ is maximized on $[0,1]$ at $x= \frac{k}{k+\l}$. 
	Since $g'(\frac{k-1}{k+\l-1}) =   \frac{(k-1)^{k-1}\l^\l}{(k+\l-1)^{k+\l-1}}$ and $g'(\frac{k}{k+\l}) = 0$, we know $f(x)$ is differentiable.
	As $g'(x)$ is decreasing on $[\frac{k}{k+\l} - \frac{\sqrt{k\l}}{k+\l}, \frac{k}{k+\l} + \frac{\sqrt{k\l}}{k+\l}]$, 
	with $\frac{k}{k+\l} - \frac{\sqrt{k\l}}{k+\l} \leq \frac{k-1}{k+\l-1}$ and $\frac{k}{k+\l} + \frac{\sqrt{k\l}}{k+\l} \geq \frac{k}{k+\l}$, 
	it follows that $f(x)$ is concave and non-decreasing. Then since $g(x)$ is maximized on $[0,1]$ at $x= \frac{k}{k+\l}$, we have
	$f(x)\geq g(x)$ for $x \in [\frac{k}{k+\l} - \frac{\sqrt{k\l}}{k+\l}, 1]$. As $f(0)=g(0)$ and $g(x)$ is convex  for $x \in [0, \frac{k}{k+\l} - \frac{\sqrt{k\l}}{k+\l}]$, it further follows that $f(x)\geq g(x)$  for  $x \in [0,\frac{k}{k+\l} - \frac{\sqrt{k\l}}{k+\l}]$.	 
\end{proof}

\begin{lem}\label{l:f2}	For $\x \leq 1/2$ the function $$ f_*(d) = \x(1-\x)^{k+\l} f(d) + \frac{(k-1)^{k-1}\l^\l}{(k+\l-1)^{k+\l-1}}(1-\x)\x^{k+\l}(1-d)$$
	achieves its maximum on $[0,1]$ for some $d \in  \left[\frac{k-1}{k+\l-1},\frac{k}{k+\l}\right]$.
\end{lem}
\begin{proof}
	We have $$f_*(d)=\begin{cases} \frac{(k-1)^{k-1}\l^\l}{(k+\l-1)^{k+\l-1}}\left(\x(1-\x)^{k+\l} - (1-\x)\x^{k+\l}\right)\cdot d + \frac{(k-1)^{k-1}\l^\l}{(k+\l-1)^{k+\l-1}}(1-\x)\x^{k+\l}&\mbox{if}\; d \in \left[0,\frac{k-1}{k+\l-1}\right],  \\
	\frac{k^k\l^\l}{(k+\l)^{k+\l}}\x(1-\x)^{k+\l} + \frac{(k-1)^{k-1}\l^\l}{(k+\l-1)^{k+\l-1}}(1-\x)\x^{k+\l}(1-d) &\mbox{if}\; d \in \left[\frac{k}{k+\l},1\right]. \end{cases}$$
	Thus $f_*(d)$ increases on $[0,\frac{k-1}{k+\l-1}]$ and decreases on $[\frac{k}{k+\l},1]$, which implies the conclusion. 
\end{proof}

\section{Lower bound and the main lemma}\label{s:lower}

In this section we take the first steps in the proof of \cref{t:main2}, reducing it to \cref{l:upper}, the proof of which occupies the rest of the paper. First, we establish the lower bound.

\begin{lem}\label{l:lower} Let $k \geq \ell \geq 1$ be integers, then
	$$ \sup_G s_{k,\l}(G) \geq \brm{OPT}(k,\l).$$ 
\end{lem}

\begin{proof}
	Suppose first $k\geq \l+1$, and let $\alpha,d$ achieve the maximum in the definition of $\brm{OPT}(k,\l)$. Let $G=G(n)$ be a random digraph on $n$ vertices defined as follows. Let $(X,Y_1,Y_2)$ be a partition of $V(G)$ such that \begin{align*} &|X|=(1+o(1))\alpha n,\\ &|Y_1|=(1+o(1))\frac{k+\l-1}{k-1}(1-d)(1-\alpha)n,\\ &|Y_2|=(1+o(1))\s{1-\frac{k+\l-1}{k-1}(1-d)}(1-\alpha)n, \end{align*}	The digraph $G$ is the complete bipartite digraph with bipartition $(X, Y_1 \cup Y_2)$ such that all the arcs between $X$ and $Y_2$ are directed towards $Y_2$, and the direction of each arc between $X$ and $Y_1$ is chosen independently at random so that it is directed towards $Y_1$ with probability $\frac{\l}{k+\l-1}$. A straightforward calculation shows that almost surely every vertex in $X$ has outdegree $$(1+o(1))\s{|Y_2|+\frac{\l}{k+\l-1}|Y_1|} = (1+o(1))d(1-\alpha) n
	$$ and indegree $(1+o(1))(1-d)(1-\x)n$. Thus the probability that the map $\phi \in \Phi$ lies in $\mc{S}$ and satisfies $\phi(c) \in X$, $\phi(I \cup O) \in Y_1 \cup Y_2$ is $$(1-o(1))\alpha(1-\alpha)^{k+\l}d^k(1-d)^\ell.$$ 
	Meanwhile, almost surely every vertex in $Y_1$ has outdegree $(1+o(1))\frac{k-1}{k+\l-1}\x n$ and indegree $(1+o(1))\frac{\l}{k+\l-1}\x n$.
	Thus the probability that the map $\phi \in \Phi$ lies in $\mc{S}$ and satisfies $\phi(c) \in Y_1$, $\phi(I \cup O) \in X$ is
	$$(1-o(1)) \frac{(k-1)^{k}\l^\l}{(k+\ell-1)^{k+\l}}\alpha^{k+\l}|Y_1|=(1-o(1))\frac{(k-1)^{k-1}\l^\l}{(k+\ell-1)^{k+\l-1}}(1-\alpha)\alpha^{k+\l}(1-d).$$
	Adding these bounds, we conclude that almost surely $s(G) \geq (1-o(1))\brm{OPT}(k,\l)$, as desired.
	
	In the case $k=\l$, the construction is simpler. We define $G$ to be a complete bipartite digraph with bipartition $(X, Y)$ such that $|X|=(1+o(1))\x$  and the direction of each arc between $X$ and $Y$ is chosen independently at random with both directions having probability $1/2$. Analogously to the previous case we  have $s(G) \geq (1-o(1))\brm{OPT}(k,\l)$ almost surely.
\end{proof}

By \cref{l:lower} it remains to show that $s_{k,\l}(G)  \leq \brm{OPT}(k,\l)$ for every digraph $G$. 

The hard part of the proof of the upper bound consists of proving the following statement.

\begin{lem}\label{l:upper} Let $k, \ell$ be as in \cref{t:main2} and let $G$ be a digraph such that $s(v) \geq \brm{OPT}(k,\l)$ for every $v \in V(G)$. Then $s(G) \leq  \brm{OPT}(k,\l).$
\end{lem}

In the remainder of this section we recall the standard argument that justifies making the assumption made in \cref{l:upper} that $s(v)$ is large for every $v \in V(G)$, and  thus deriving the upper bound in \cref{t:main2} from this lemma.

\begin{proof}[Proof of \cref{t:main2} modulo \cref{l:upper}.]
	As we already noted, by \cref{l:lower}, it suffices to show that $s(G)  \leq \brm{OPT}(k,\l)$ for every digraph $G$. Suppose not that there exists some $G$ such that $\delta = s(G)  - \brm{OPT}(k,\l) > 0$. We first show that we may assume that $|V(G)|$ is large.
	Indeed, for any integer $t \geq 1$ consider the blowup $G^{(t)}$ obtained by replacing every $v \in V(G)$ by an independent set $U_v$ with $|U_v|=t$, and for every arc  $vv' \in E(G)$ and all $u \in U_v$ and $u' \in U_{v'}$ we add an arc $uu'$ to $G^{(t)}$. It is easy to see that $s(G^{(t)}) \geq s(G)$. Therefore we may
	replace $G$ by $G^{(t)}$ for $t$ sufficiently large and assume that $n=|V(G)|> (k+\l)/\delta$. Moreover, we assume without loss of generality that $s(G) \geq s(G')$ for every $G'$ with $|V(G')| = n$.
	
	By \cref{l:upper}, there exists $v \in V(G)$ such that $s(v) < \brm{OPT}(k,\l)$.  Meanwhile 
	by \eqref{eq:sG} there exists $u \in V(G)$ such that $s(u) \geq s(G) > s(v)+\delta$. Let $G'$ be a digraph obtained from $G$ by deleting $v$ and adding a vertex $u'$ which has exactly the same in-neighbors and out-neighbors as $u$ in $G \setminus v$.
	Then
	\begin{align*}
	s(G)&\ge s(G') \ge s(G)+\frac{(k+\l+1)s(u)}{n}-\frac{(k+\l+1)s(v)}{n}-\frac{(k+\l+1)(k+\l)s(u,v)}{n^2}\\
	&\ge s(G)+\frac{k+\l+1}{n}\left(\delta - \frac{k+\l}{n}\right) > s(G),
	\end{align*}
	a contradiction. This proves \cref{t:main2} (assuming \cref{l:upper}).
\end{proof}

\section{Proof of \cref{l:upper}}\label{s:upper}

\subsection{Further notation and first estimates}

For brevity, let \begin{align*}m=k+\l,%&S=(m+1)\brm{OPT}(k,\l), \\&
\qquad\lambda_0 = \frac{k^k \l^\l}{m^{m}} \quad \mathrm{and} \quad\lambda_1 = \frac{k^k \l(\l-1)^{\l-1}}{(m-1)^{m-1}}.\end{align*}  

By \cref{l:xayb} we have 
$$x^ky^\l \leq \lambda_0(x+y)^m$$
for all $x,y \geq 0$, and we frequently employ \cref{l:xayb} in a similar manner.

We start by establishing a bound on $s(v)$ in terms of $\d(v)$. This will imply that $\d(v)$ is either close to one or to zero, allowing us to define the bipartition of $G$ accordingly.

\begin{claim}\label{c:degree} For every $v \in V(G)$ we have
	\begin{equation}\label{eq:sx}
	s(v) \le \frac{1}{m+1}\left(\lambda_0  \d^{m}(v) + \lambda_1 \d(v)(1 - \d(v))^{m-1}\right).
	\end{equation}
\end{claim}

\begin{proof}
Any $S_{k,\ell}$ containing $v$ as its center has $\ell$ vertices from $N^-(v)$ and $k$ vertices from $N^+(v)$. So we have
\begin{align}\label{eq:sxcenter}
s(c\to v)\le \d^+(v)^k \d^-(v)^\l \le
\lambda_0\d^{m}(v).
\end{align}
If $\phi \in \mc{S}$ maps a leaf of $\SK$ to $v$ 
then the center of this $\SK$ is mapped to a neighbor of $v$ and other leaves are mapped to non-neighbors of $v$.
Thus using  \cref{l:xayb}, we have
\begin{align}\label{eq:sxo}
s(o \to v) &\le\d^-(v) \left[(1-\d(v))^{m-1}\left(\frac{\l}{m-1}\right)^{\l}\left(\frac{k-1}{m-1}\right)^{k-1}\right],
\end{align}
and symmetrically,
\begin{align}\label{eq:sxi}
s(i \to v) &\le \d^+(v) \left[(1-\d(v))^{m-1}\left(\frac{\l-1}{m-1}\right)^{\l-1}\left(\frac{k}{m-1}\right)^{k}\right].
\end{align}
From \eqref{eq:sxo} and \eqref{eq:sxi} we derive
\begin{align}\label{eq:sxio}
\l \cdot s(i \to v) + k\cdot s(o \to v)  &\le \frac{k^k \l^{\l}}{(m-1)^{m-1}}\left({1 - \d(v)}\right)^{m-1}  \left(\d^-(v)\left(\frac{k-1}{k} \right)^{k-1} +  \d^+(v)\left(\frac{\l-1}{\l} \right)^{\l-1}\right) \notag \\ & \le \lambda_1\d(v)(1 - \d(v))^{m-1}. 
\end{align}
Plugging \eqref{eq:sxcenter} and \eqref{eq:sxio} into \eqref{eq:sgv}  we obtain \eqref{eq:sx}.
\end{proof}

Let $X$ be the set of vertices $v \in V(G)$ such that $\rho(v) \geq 1/2$, and let $Y=V(G)\setminus X$. Our goal is to show that any deviation of $G$ from the family of examples described in \cref{l:lower} reduces  $s(G)$. To facilitate our analysis we now partition $\mc{S}$ into several subsets. We say that a map $\phi \in \mc{S}$ is \begin{itemize}
	\item \emph{Type $1$}, if $\phi (c) \in X$, $\phi(I \cup O) \subseteq Y$,
	\item \emph{Type $2$}, if $\phi (c) \in Y$, $\phi(I \cup O) \subseteq X$,
	\item \emph{Type $X$}, if $\Im(\phi) \subseteq X$,
	\item \emph{Type $Y$}, if $\Im(\phi) \subseteq Y$,
	\item \emph{Type $0$}, otherwise, i.e. when  $\phi(I \cup O) \cap X \neq \emptyset$ and $\phi(I \cup O) \cap Y \neq \emptyset$.
\end{itemize} 
For a type $T \in \{1,2,X,Y,0\}$, let $\mc{S}_T$ denote the set of maps in $\mc{S}$ of type $T$, let $s_T(G) = \Pr[\phi \in \mc{S}_T]$, and define $s_T(v), s_T(c \to v)$, etc., accordingly, e.g.
$s_T(v)=\Pr[\phi \in \mc{S}_T | v \in \Im(\phi)].$

Our estimates are given in terms of the following parameters of $G$: \begin{align*} &S= (m+1)\min_{v \in V(G)}s(v) \geq 
(m+1)\brm{OPT}(k,\l),  &D = \min_{x \in X}\d(x), \\ &\x = \mu(X)=|X|/|V(G)|, 
&\y=\max_{y \in Y}\d_Y(y), \\&S_1 = \min_{x \in X} (\d^{+}_Y(x))^k(\d^{-}_Y(x))^\l, &\gamma =\min_{x \in X}\rho^{-}_Y(x). \end{align*}

First, we upper bound the probability of a star with a leaf in $X$, given a fixed center. 
\begin{claim}\label{c:leafx}
	\begin{align}
	\Pr[\phi \in \mc{S} \wedge (\phi(I \cup O) \cap X  \neq \emptyset) | \phi(c) = v] \leq \lambda_1 \rho_X(v) (1-D)^{m-1}. \label{e:leafx}
	\end{align}   
\end{claim}	

\begin{proof}
	Using a variation of the estimates used in the proof of \cref{c:degree} we obtain
	\begin{align*}
	\Pr&[\phi \in \mc{S} \wedge (\phi(I \cup O) \cap X  \neq \emptyset) | \phi(c) = v]  \\
	&\leq \alpha \bb{E}_{x \in X}[\l \cdot s(i \to x, c \to v) + k\cdot s(o \to x, c \to v)] \\
	&\leq (1-D)^{m-1} \s{\l\frac{k^k(\l-1)^{\l-1}}{(m-1)^{m-1}}\rho^-_X(v) + k\frac{(k-1)^{k-1}\l^{\l}}{(m-1)^{m-1}}\rho^+_X(v) } \\ &\leq \lambda_1 \rho_X(v) (1-D)^{m-1}, 
	\end{align*}  
	where the last inequality holds because $\left(\frac{k-1}{k}\right)^{k-1}\leq \left(\frac{\l-1}{\l}\right)^{\l-1}$, as desired. 
\end{proof}

\begin{claim}\label{c:AM} 
\begin{align} 
	&\lambda_0(1-\x)^{m} \geq \gamma^\l(1-\x - \gamma)^k  \geq  S_1 \geq S -  \lambda_1(1+\x)(1-D)^{m-1}. \label{e:alpha} 
	\end{align}
\end{claim}
\begin{proof} Note that if $X  =\emptyset$ the claim trivially holds, and so we assume $X \neq \emptyset$.
	
	We have $$\lambda_0(1-\x)^{m} \geq \gamma^\l(1-\x - \gamma)^k  \geq  S_1,$$ where the first inequality holds by \eqref{e:xayb}, and the second by definition of $S_1$.
	
	 To verify the last inequality in \eqref{e:alpha} choose $x \in X$ such that $S_1=(\d^{+}_Y(x))^k(\d^{-}_Y(x))^\l$. Then $s_1(c \to x)~\leq~S_1,$
	 and
	 $$\l \cdot s(i \to x) + k\cdot s(o \to x) \leq \lambda_1 (1-D)^{m-1},$$
	 by \eqref{eq:sxio}. Using  \eqref{e:leafx} to bound the remaining contribution to $s(x)$,  we obtain
	 \begin{align*}
	    S   \leq (m+1)s(x) &=s(c\to x) + \l \cdot s(i \to x) + k\cdot s(o \to x)\\
	    &\leq S_1+\lambda_1 \rho_X(v) (1-D)^{m-1} + \lambda_1 (1-D)^{m-1}\\
	    &\leq S_1 + \lambda_1(1+\alpha)(1-D)^{m-1}
	 \end{align*}
	 implying \eqref{e:alpha}. 
\end{proof}	

It is easy to derive from \cref{c:AM} that $\alpha \neq 1$, i.e. $Y \neq \emptyset$. (For a stronger estimate on $\alpha$, see \cref{c:alphaU} below.)

\begin{claim}\label{c:B} Either
	\begin{align} 
&\lambda_1  \alpha \left( (1-\alpha - \beta)^{m-1} + (1-D)^{m-1}\right) +(\lambda_0+\lambda_1) \beta^{m}  \geq S, \qquad \mathrm{or}  \notag\\
&(\lambda_0+\lambda_1) \beta^{m} +  \frac{m(m-1)}{\binom{m}{k}}\alpha  (\alpha + 
\beta)(1-D)^{m-2} \geq S. \label{e:beta} 
%	(m+1)\lambda_0 (\alpha + (1-\alpha)(\alpha+\beta)^m) & \geq S. \label{e:lalpha} 
	\end{align}	
\end{claim}
\begin{proof}
    Let $y \in Y$ be chosen so that $\d_Y(y)=\beta$, and let $\alpha' = \rho_X(y)$. Then 
	$$s_1(y) \leq \frac{1}{m+1} \lambda_1  \alpha' (1-\alpha - \beta)^{m-1}$$
Meanwhile, repeating the estimates in the proof of \cref{c:degree} gives
$$\Pr[(\phi \in \mc{S}) \wedge (\Im(\phi) \subseteq Y) | y \in \Im(\phi)] \leq \frac{1}{m+1} (\lambda_0+\lambda_1) \beta^{m}.$$ Next we show that
\begin{equation}
\label{e:B2}\Pr[(\phi \in \mc{S}) \wedge  (\phi(I \cup O) \cap X \neq \emptyset) |  y \in \phi(I \cup O)] \leq \frac{1}{\binom{m}{k}}(m-1)(\alpha -\alpha') (\alpha' + 
\beta) (1-D)^{m-2}.
\end{equation}
To prove \eqref{e:B2} we first upper bound the probability that a map $\phi \in \Phi$ with $y \in \phi(I \cup O)$ satisfies  $\phi(I \cup O) \cap X \neq \emptyset$ and maps edges of $\SK$ to edges of $G$ and non-edges to non-edges, while not necessarily preserving edge directions. For each of $m-1$ vertices in $(I \cup O) - \{\phi^{-1}(y)\}$, the probability that this vertex is mapped to a non-neighbor of $y$ in $X$ is $\alpha - \alpha'$. The probability that $c$ is mapped to a neighbor of $y$ is $\rho(y)  \leq \alpha +\beta$. Finally, the remaining vertices in $I \cup O$ must be mapped to non-neighbors of the chosen leaf in $X$ which happens with probability at most $(1-D)^{m-2}$. This yields the upper bound
$(m-1)(\alpha-  \rho_X(y)) (\rho_X(y) + 
\beta) (1-D)^{m-2}$ on the probability that $\phi$ has the above properties.
As the probability that such a map $\phi$ preserves directions is $\frac{1}{\binom{m}{k}}$, \eqref{e:B2} follows.
Combining the above inequalities and \eqref{e:leafx} we obtain
	\begin{align*} &S \leq (m+1)s(y) \\ &\leq  \lambda_1  \alpha' \left( (1-\alpha - \beta)^{m-1} + (1-D)^{m-1}\right) +(\lambda_0+\lambda_1) \beta^{m} + \frac{m(m-1)}{\binom{m}{k}}(\alpha -\alpha') (\alpha + 
\beta)(1-D)^{m-2}. \end{align*}
As the last term is a linear function of $\alpha'$, the inequalities hold either for $\alpha' = \alpha$ or $\alpha' = 0$, implying \eqref{e:beta}.
\end{proof}

\begin{claim}\label{c:AL} 
	\begin{equation} \label{e:AL}
	\lambda_0  \x (1-\x)^{m} + \lambda_1 \x (1-D)^{m-1} + \lambda_0(1-\x)\y^m \geq \frac{S}{m+1}.
	\end{equation} 	
\end{claim}

\begin{proof}
Note that 	for any $v \in V(G)$
\begin{equation}\label{e:szv}
\Pr[ (\phi \in \mc{S}) \wedge (\phi(I \cup O) \subseteq Y) |\phi(c)=v] \leq \lambda_0 \rho^{m}_Y(v), 
\end{equation}
	Using this and \eqref{e:leafx} we obtain
	\begin{align*}
	 \frac{S}{m+1} &\leq s(G) \leq \bb{E}_{v \in V(G)}[s(c \to v)] \\&= \bb{E}_{v \in V(G)}\s{\Pr[\phi \in \mc{S} \wedge (\phi(I \cup O) \subseteq Y) | \phi(c) = v]}  \\ &+ \bb{E}_{v \in V(G)}\s{\Pr[\phi \in \mc{S} \wedge (\phi(I \cup O) \cap X  \neq \emptyset) | \phi(c) = v]} \\ &\leq \alpha\bb{E}_{x \in X}[s_1(c \to x)] + (1-\x)\bb{E}_{y \in Y}[s_Y(c \to y)] + \lambda_1 \x (1-D)^{m-1}\\ &\leq 	\lambda_0  \x (1-\x)^{m}  + \lambda_0(1-\x)\y^m + \lambda_1 \x (1-D)^{m-1},
	\end{align*}
	as desired.
\end{proof}

The inequalities derived in this section can be used to derive fairly precise bounds on $D,\x,\y,\gamma$ and $S_1$ with errors decaying exponentially as $m$ grows. The following claims are the consequences of such estimates needed in the subsequent analysis. The proofs of these claims are neither short, nor especially difficult or inspiring, and are thus relegated to Appendix \ref{sec:claim}.

\begin{claim}\label{c:alphaU}  
	$\alpha \leq \frac{1}{m}.$	
\end{claim}

\begin{claim}\label{c:gamma} 
	$\gamma \geq \x/2.$
\end{claim}

Let 
\begin{equation}\label{e:s2} 
S_2 = S_1 - \ \frac{(m-1)}{(m+1)\binom{m}{k}}(\alpha+\beta)(1-\alpha)(1-D)^{m-2}.
\end{equation}

\begin{claim}\label{c:S2} 
	$S_2 \geq \lambda_1 \frac{1 - \alpha}{\alpha}\beta^{m-1}.$
\end{claim}

\begin{claim}\label{c:final} If $\l=1$, then
	\begin{equation}\label{e:final}
	\frac{\x^{k+1}}{k^k}-\frac{\x^{k+1}}{(k+1)^{k+1}}\frac{\x}{\gamma} \geq  \frac{\y^{k+1}}{2k^k}.
	\end{equation}
\end{claim}

\subsection{Contribution of non adjacencies between parts}

We now embark on our quest of deriving an upper bound on $s(G)$, which we obtain by bounding $s_1(G),s_2(G),s_X(G),s_Y(G)$ and $s_0(G)$ separately. First we show that the missing edges between $X$ and $Y$ lead to more losses in $s_1(G)$ than gains in $s_0(G)$. 

\begin{claim}\label{c:nonedge} For a pair of non-adjacent $x \in X, y\in Y$ we have
	$$ s(x,y) \leq \frac{m-1}{(m+1)\binom{m}{k}}(\alpha+\beta)(1-D)^{m-2}. $$ 
\end{claim}
\begin{proof}
First, as in the proof of \cref{c:B}, we bound from above the probability that a map $\phi \in \Phi$ with $x,y \in \Im(\phi)$ maps the edges of $S_{k,\l}$ to edges of $G$ (possibly reversing the directions) and non-edges to non-edges. Thus preimages of both $x$ and $y$ in $\phi$ must be leaves of $\SK$, which happens with probability $\frac{m-1}{m+1}$. The center must be mapped to a neighbor of $y$, which happens with probability $\rho(y) \leq \alpha+\beta$, and the remaining $m-2$ leaves are mapped to non-neighbors of $x$, which happens with  probability at most $(1-D)^{m-2}$, subject to the previous constraints. Thus the probability that $\phi$ has above properties is upper bounded by $$\frac{m-1}{m+1}(\alpha+\beta)(1-D)^{m-2}.$$ 
The probability that the map $\phi$ as above preserves the edge directions is exactly $\frac{1}{\binom{m}{k}}$, yielding the claimed bound on $s(x,y)$.
\end{proof}

Let $$d_0 =\frac{1}{n^2}|\{(x,y) | x \in X, y\in Y, xy,yx \not \in E(G)\}|,$$
denote the probability that a random pair of vertices corresponds to a non-adjacent pair of vertices where the first one is in $X$ and the second one is in $Y$.

\begin{claim}\label{c:nonedge2} We have
	$$ s_0(G) \leq \frac{(m-1)m}{(m+1)\binom{m}{k}}(\alpha+\beta)(1-D)^{m-2}d_0.$$ 
\end{claim}
\begin{proof} We consider selecting a uniformly random map $\phi \in \Phi$ as follows. First, we select vertices $(v_1,v_2,\ldots,v_{m+1})$  independently and uniformly at random. Then we select a random permutation $(u_1,u_2,\ldots,u_{m+1})$ of $V(S_{k,\l})$, and let $\phi(u_i) = v_i$. The probability that the resulting map is in $\mc{S}_0$ is $s_0(G)$. As every map in  $\mc{S}_0$ contains at least $m-1$ pairs of leaves such that the first one is mapped to $X$ and the second one in $Y$, we have that with probability at least $s_0(G)\frac{m-1}{m(m-1)}=s_0(G)/m$, the resulting map satisfies $v_1 \in X, v_2 \in Y$ and $v_1,v_2$ are non-adjacent. On the other hand, by \cref{c:nonedge} this probability is at most $$ \frac{m-1}{(m+1)\binom{m}{k}}(\alpha+\beta)(1-D)^{m-2}\cdot d_0,$$ 
 implying the claim.  
\end{proof}

\begin{claim}\label{l:bound2l} For every $x \in X$ we have
	\begin{equation}\label{e:six0}
s_1(c \to x) \leq (1-\alpha)^m f\s{1-\frac{\d^{-}_{Y}(x)}{1-\alpha}} - \frac{m\d^0_Y(x)}{(1-\alpha)}S_1.
	\end{equation}	
\end{claim}
\begin{proof}
	By \cref{l:f} the function $f$ is non-decreasing and so we have
	$$f\s{1-\frac{\d^{-}_{Y}(x)}{1-\alpha}} = f\s{\frac{\d^{+}_{Y}(x)+\d^{0}_Y(x)}{1-\alpha}} \geq f\s{ \frac{\d^{+}_Y(x)}{1 -\alpha} \s{1+\frac{\d^0_Y(x)}{\d^+_Y(x)+\d^-_Y(x)} } }. $$
As $$ 1 - \frac{\d^{+}_Y(x)}{1 -\alpha} \s{1+\frac{\d^0_Y(x)}{\d^+_Y(x)+\d^-_Y(x)} } = \frac{\d^{-}_Y(x)}{1 -\alpha} \s{1+\frac{\d^0_Y(x)}{\d^+_Y(x)+\d^-_Y(x)} },  $$
using the lower bound on $f$ from  \cref{l:f} we further have
 $$ (1-\x)^mf\s{ \frac{\d^{+}_Y(x)}{1 -\alpha} \s{1+\frac{\d^0_Y(x)}{\d^+_Y(x)+\d^-_Y(x)} } } \geq  \s{1 + \frac{\d^0_Y(x)}{\d^+_Y(x)+\d^-_Y(x)}}^m  (\d^{+}_Y(x))^k(\d^{-}_Y(x))^\l.$$
Using the above lower bound on the second term of \eqref{e:six0} and the inequalities $s_1(c \to x) \leq (\d^{+}_Y(x))^k(\d^{-}_Y(x))^\l$, $S_1 \leq (\d^{+}_Y(x))^k(\d^{-}_Y(x))^\l$ for the other two terms, and dividing by $(\d^{+}_Y(x))^k(\d^{-}_Y(x))^\l$, we reduce \eqref{e:six0} to
$$1 \leq  \s{1 + \frac{\d^0_Y(x)}{\d^+_Y(x)+\d^-_Y(x)}}^m  - \frac{m\d^0_Y(x)}{(1-\alpha)},$$
which holds as $\d^+_Y(x)+\d^-_Y(x) \leq 1-\alpha$.
\end{proof}

Let $$d = 1-\bb{E}_{x \in X }\s{\frac{\d^{-}_{Y}(x)}{1-\alpha}} = 1 - \bb{E}_{y \in Y }\s{\frac{\d^{+}_{X}(y)}{\alpha}}.$$
Thus $d$ is the probability that choosing a uniformly random $x \in X$ and a uniformly random $y \in Y$ we have $yx \not \in E(G)$.

\begin{claim}\label{c:bound2l} 
	\begin{equation}\label{e:s1i}
	s_1(G) \leq \alpha(1-\alpha)^m f(d) - \frac{m}{1 -\alpha} S_1 d_0.
	\end{equation}	
\end{claim}
\begin{proof} Note that $d_0 = \alpha \bb{E}_{x \in X} \rho^0_Y(x)$. Using \eqref{e:six0} and the concavity of $f$ we have 
\begin{align*}
s_1(G) &= \bb{E}[s_1(c \to v)] = \alpha \bb{E}_{x \in X}[s_1(c \to x)]   \\
& \le \alpha (1-\alpha)^m \bb{E}_{x \in X}\left[f\s{1-\frac{\d^{-}_{Y}(x)}{1-\alpha}}\right] - \frac{m \alpha\bb{E}_{x \in X} \d^0_Y(x)}{(1-\alpha)}S_1 \\ &\leq \alpha (1-\alpha)^m f(d) - \frac{m}{1 -\alpha}S_1 d_0,
\end{align*}
as desired.
\end{proof}	

\subsection{Contribution of edges in $Y$}

Our next goal is to show that the edges in $Y$ contribute more losses to $s_1(G)$ than gains to $s_Y(G)$.

For $Z \subseteq V(G)$, let  $\d(Z)=\frac{|E(G[Z])|} {\binom{|Z|}{2}}$ denote the density of edges in $Z$.

\begin{lem}\label{l:baddensity} If $\l \geq 2$ then
	for any $Z  \subseteq V(G)$ and any $v \in V(G)\setminus Z$ we have $$\Pr[(\phi \in \mc{S}) \wedge \phi(I \cup O) \subseteq Z |  \phi(c)=v] \leq (1-\d(Z))(\mu(Z) - \rho^{-}_Z(v))^{k}(\rho^{-}_Z(v))^{\l}\leq (1-\d(Z))\lambda_0 (\mu(Z))^m.$$
\end{lem}
\begin{proof}
	Let $Z_1 = N^{-}_Z(v)$, $Z_2 = Z -Z_1$. Then either \begin{itemize} \item 
		at least $\d(Z)\binom{|Z_1|}{2}$ pairs of vertices in $G_1$ are adjacent, or \item at least  $\d(Z)\binom{|Z_2|}{2}$ pairs of vertices in $Z_2$ are adjacent,  or \item at least $\d(Z)|Z_1||Z_2|$ pairs $(z_1,z_2)$ such that $z_1 \in Z_1,z_2 \in Z_2$ are adjacent.
	\end{itemize}
Suppose the first case holds. (The remaining cases are analogous.)
We choose $\phi$ by first choosing images of two vertices in $I$. They must be mapped to a pair of non-adjacent vertices in $Z_1$  which happens with probability at most $(1-\d(Z))(\rho^{-}_Z(v))^{2}$. The remaining vertices in $O$ are mapped to $Z_1$, while the vertices in $I$ are mapped to $Z_2$ with probability $(\mu(Z) - \rho^{-}_Z(v))^{k}(\rho^{-}_Z(v))^{l-2}$, giving the desired bound.
\end{proof}

%\begin{equation}\label{e:szuv}
%\Pr[ (\phi \in \mc{S}) \vee (\phi(I \cup O) \in Z) |\phi(o)=u, \phi(c)=v] \leq %\frac{(k-1)^{k-1}\l^\l}{(m-1)^{m-1}}\min\{(\mu(Z)-\rho_Z(v))^{m-1},\rho^{m-1}_Z(u)\}. 
%\end{equation}

\begin{lem}\label{l:gooddensity}
	For every $Z  \subseteq V(G)$ we have   $$\Pr[(\phi \in \mc{S}) \wedge (\Im(\phi) \subseteq Z) ] \leq \frac{(k-1)^{k-1}\l^\l}{(m-1)^{m-1}}\mu^2(Z)\cdot \frac{\d(Z)}{2} \cdot \max_{z \in Z}\rho^{m-1}_Z(z).$$
\end{lem}
\begin{proof}
	 We have $$\Pr[(\phi(c), \phi(o) \in Z) \wedge (\phi(c)\phi(o) \in E(G))] = \frac{|E(G[Z])|}{n^2} \leq \mu^2(Z)\frac{\d(Z)}{2}.$$ 
For any such choice of $\phi(c)$ and $\phi(o)$, the probability that the remaining vertices are appropriately mapped to in- and out-neighbors of $c$ is $$ (\rho^+_Z(\phi(c))^{k-1}(\rho^-_Z(\phi(c))^{\l} \leq \frac{(k-1)^{k-1}\l^\l}{(m-1)^{m-1}}\rho^{m-1}_Z(\phi(c)),$$
which implies the claimed bound. 
\end{proof}

\begin{claim}\label{c:sy} For $\l \geq 2$ and any $x \in X$ we have
	\begin{equation}\label{e:s1ylocal}
	s_1(c \to x) \leq  (1-\alpha)^m f\s{1-\frac{\d^{-}_{Y}(x)}{1-\alpha}} - S_1\d(Y).
	\end{equation}	
\end{claim}
\begin{proof}
	 Using \cref{l:baddensity} we obtain
\begin{align*} s_1(c \to x) &\leq (1-\d(Y))(1-\alpha - \rho^{-}_Y(x))^{k}(\rho^{-}_Y(x))^{\l}  \\
&\leq (1-\alpha)^m f\s{1-\frac{\d^{-}_{Y}(x)}{1-\alpha}} - (\rho^{+}_Y(x))^{k}(\rho^{-}_Y(x))^{\l}\d(Y),
\end{align*}
which gives the claimed bound by definition of $S_1$.
\end{proof}

Analogously to the proof of \cref{c:bound2l}, averaging \eqref{e:s1ylocal} over $x \in X$ we obtain the following.

\begin{claim}\label{c:s1y} For $\l \geq 2$ we have
	\begin{equation}\label{e:s1y}
	s_1(G) \leq  \alpha(1-\alpha)^m f(d) -\alpha S_1\d(Y).
	\end{equation}	
\end{claim}

%We also need a technical lower bound on $S_1$ the proof of which is given in the appendix.

%\begin{claim}\label{c:S11} For $\l \geq 2$ 	\begin{equation}\label{e:condition1} 
%	S_1 \geq \frac{(k-1)^{k-1}{\l}^{\l}}{2(m-1)^{m-1}\alpha} (1-\alpha)^2\y^{m-1} +  \frac{(m-1)}{(m+1)\binom{m}{k}}(\alpha+\beta)(1-\alpha)(1-D)^{m-2}
%	\end{equation} 
%\end{claim}

Combining the claims proved in this section, we now bound $s_1(G) + s_Y(G) + s_0(G)$.

\begin{claim}\label{c:total0} If $\l \geq 2$ then
		\begin{equation}\label{e:total0} 
s_1(G) + s_Y(G) + s_0(G) \leq  \alpha(1-\alpha)^m f(d).
\end{equation}
\end{claim}	

\begin{proof}
	By \cref{l:gooddensity} we have 
\begin{equation}\label{e:sy}
		s_Y(G) \leq \frac{(k-1)^{k-1}{\l}^{\l}}{(m-1)^{m-1}} (1-\alpha)^2\frac{\d(Y)}{2} \y^{m-1}. 
\end{equation}
Let $$\eps = \frac{1}{S_1}\cdot\frac{(m-1)}{(m+1)\binom{m}{k}}(\alpha+\beta)(1-\alpha)(1-D)^{m-2}.$$
By \cref{c:S2} we have
$$(1 -\eps)S_1 = S_2 \geq \lambda_1\frac{1-\x}{\x}\beta^{m-1} \geq \frac{(k-1)^{k-1}{\l}^{\l}}{(m-1)^{m-1}} \frac{(1-\alpha)^2}{\x}\frac{ \y^{m-1}}{2}.$$
In particular $\eps \in [0,1]$.
Thus taking a convex combination of \eqref{e:s1i} and (\ref{e:s1y})we get
	\begin{align*}
	s_1(G) &\leq  \alpha(1-\alpha)^m f(d) - \eps\frac{m}{1 -\alpha} S_1 d_0  -  (1-\eps)\alpha S_1 \d(Y)\\ &\leq
	\alpha(1-\alpha)^m f(d) - \frac{m(m-1)}{(m+1)\binom{m}{k}}(\alpha+\beta)(1-D)^{m-2} d_0 - \frac{(k-1)^{k-1}{\l}^{\l}}{2(m-1)^{m-1}}(1-\alpha)^2\y^{m-1}\d(Y)  \\ &\leq
	\alpha(1-\alpha)^m f(d) - s_0(G) - s_Y(G),
	\end{align*}	
 where the last inequality follows from \cref{c:nonedge} and \eqref{e:sy}.
\end{proof}

\subsection{The case $\l\geq 2$}
In this subsection, we complete the proof of \cref{l:upper} for the case $\l\geq 2$.
In view of \cref{c:total0}, it remains to bound $s_2(G)$ and $s_X(G)$ 
(see \cref{c:s2} below, which in fact works for $\l\geq 1$).

First, to dispatch the easier case $k = \l$, we use the following easy and rather weak consequence of \cref{c:alphaU}.

\begin{claim}\label{c:alpha} If $k = \l$, then
	\begin{equation}\label{e:condition2}
	\frac{(k-1)^{k-1}k^k 2^{2k-1}}{(2k-1)^{2k-1}} \leq \frac{1-\x}{\x}.
	\end{equation}		
\end{claim}

\begin{proof}
	We have
	\begin{align*}
	\frac{(k-1)^{k-1}k^k 2^{2k-1}}{(2k-1)^{2k-1}} = \frac{2k}{2k-1}\fs{4k(k-1)}{(2k-1)^2}^{k-1} \leq \frac{2k}{2k-1} < \frac{1 -\frac{1}{2k}}{\frac{1}{2k}} \leq \frac{1-\x}{\x},
	\end{align*}
	as desired, where the last inequality follows from \cref{c:alphaU}.
\end{proof}

We are now ready to finish the proof of \cref{l:upper} in the case $k = \l$.

\begin{proof}[Proof of \cref{l:upper} in the case $k = \l$]
	By \cref{l:baddensity} we have 
$$
		s_2 (G) \leq (1-\d(X))(1-\x)\frac{\x^m}{2^{m}},
$$
and by \cref{l:gooddensity} we have
$$
	s_X(G)  \leq  \frac{(k-1)^{k-1}k^k}{(2k-1)^{2k-1}}\alpha^{m+1}\frac{\d(X)}{2}.
$$ 
Combining these inequalities and \eqref{e:condition2}, we get
	\begin{align*} s_2(G) + s_X(G) \leq \frac{1}{2^m}(1-\x)\x^m + \d(X)\alpha^{m+1}\left( \frac{(k-1)^{k-1}k^k}{2(2k-1)^{2k-1}} - \frac{1-\x}{\x2^m}\right) \leq \frac{1}{2^m}(1-\x)\x^m. 
	\end{align*}
Meanwhile, as $f(d) \leq \frac{1}{2^m}$	for any $d \in [0,1]$, \cref{c:total0} implies
$$ s_1(G) + s_Y(G) + s_0(G) \leq  \frac{1}{2^m}\alpha(1-\alpha)^m.$$
Summing these two inequalities we obtain $$s(G) \leq \frac{1}{2^m} \s{\alpha(1-\alpha)^m + (1-\x)\x^m} \leq \brm{OPT}(k,\l),$$
as desired.
\end{proof}	

In the case $k>\l$, we need the following, more subtle estimates. 

\begin{claim}\label{c:bound0} For every $y \in Y$, $x \in N^{+}_X(y)$
	\begin{equation}\label{e:s11}
	\frac{(k-1)^{k-1}\l^\l}{(m-1)^{m-1}} \x^{m} \geq \x s_2(c \to y, o \to x) +  (m-1) s_X(o \to x). 
	\end{equation} 
\end{claim}

\begin{proof}
		We have 
		$$ s_2(c \to y, o \to x)  \leq (\d^{+}_{X - N_X(x)}(y))^{k-1}(\d^{-}_{X - N_X(x)}(y))^{\l} \leq \frac{(k-1)^{k-1}\l^\l}{(m-1)^{m-1}}(\x - \d_X(x))^{m-1}.$$ 
		Meanwhile, 
		$
		s_X(o \to x) \leq \frac{(k-1)^{k-1}\l^\l}{(m-1)^{m-1}} 
		\rho_X(x)(\x - 	\rho_X(x))^{m-1}
		. 
		$
		Substituting these upper bounds on $ s_2(c \to y, o \to x)$ and $s_X(o \to x)$, and dividing by $\frac{(k-1)^{k-1}\l^\l}{(m-1)^{m-1}}$, we reduce \eqref{e:s11} to
		$$
		\x^{m}  \geq \x(\x - \d_X(x))^{m-1} + (m-1)\d_X(x)(\x - \d_X(x))^{m-1}.
		$$
		This last inequality holds as \begin{align*} \x(\x - \d_X(x))^{m-1} &+ (m-1)\d_X(x)(\x - \d_X(x))^{m-1} =  \x^m\s{1+(m-1)\frac{\d_X(x)}{\x}}\s{1 -\frac{\d_X(x)}{\x}}^{m-1}\\ &\leq \x^m\s{1+\frac{\d_X(x)}{\x}}^{m-1}\s{1 -\frac{\d_X(x)}{\x}}^{m-1} \leq \x^m. \end{align*}
		So this claim is completed.
\end{proof}

\begin{claim}\label{c:bound1} For every $y \in Y$ we have
	\begin{equation}\label{e:s1}
	\frac{(k-1)^{k-1}\l^\l}{(m-1)^{m-1}}\x^{m-1} \d^{+}_X(y) \geq s_2(c \to y) +  \frac{ \rho^{+}_X(y)}{\gamma}\bb{E}_{x \in N^{+}_X(y)}[s_X(o \to x)]. 
	\end{equation}
\end{claim}
\begin{proof}
	Averaging \eqref{e:s11} over such $x \in  N^{+}_X(y)$ and multiplying by $\rho^{+}_X(y)/\alpha$, we obtain
	$$ \frac{(k-1)^{k-1}\l^\l}{(m-1)^{m-1}}\x^{m-1} \d^{+}_X(y) \geq s_2(c \to y) +  \frac{(m-1) \rho^{+}_X(y)}{\alpha}\bb{E}_{x \in N^{+}_X(y)}[s_X(o \to x)],$$ 
which implies  \eqref{e:s1} by \cref{c:gamma}.
\end{proof}

\begin{claim}\label{c:s2}  
	\begin{equation}\label{e:s2sx}
	\frac{(k-1)^{k-1}\l^\l}{(m-1)^{m-1}}(1-d) \x^m (1-\x) \geq s_2(G) +  s_X(G).
	\end{equation}
\end{claim}

\begin{proof}
	Averaging \eqref{e:s1} over $y \in Y$ we obtain
\begin{align*}
	\frac{(k-1)^{k-1}\l^\l}{(m-1)^{m-1}} \x^{m-1} \bb{E}_{y \in Y}(\d^{+}_X(y))  &\geq \bb{E}_{y \in Y}(s_2(c \to y)) +  \bb{E}_{y \in Y} \left(\frac{\rho^{+}_X(y)}{\gamma}\bb{E}_{x \in N^{+}_X(y)}[s_X(o \to x)]\right). 
\end{align*}
	Note that $$(1 - \alpha)\bb{E}_{y \in Y} \left(\frac{\rho^{+}_X(y)}{\gamma}\bb{E}_{x \in N^{+}_X(y)}[s_X(o \to x)]\right)  = \alpha\bb{E}_{x \in X} \s{\frac{\rho^{-}_Y(x)}{\gamma}s_X(o \to x)}.$$ And since $\bb{E}_{y \in Y}(\d^{+}_X(y)) = \alpha(1-d)$ and $\rho^{-}_Y(x) \geq \gamma$ for every $x \in X$, the above inequalities imply that
		$$
		\frac{(k-1)^{k-1}\l^\l}{(m-1)^{m-1}} \x^{m} (1-d) \geq \frac{s_2(G)}{1-\alpha} +  \frac{\alpha}{1 - \alpha} \bb{E}_{x \in X}\left(s_X(o \to x)\right). 
		$$
	Finally, note that $s_X(G) = \alpha \bb{E}_{x \in X}[s_X(o \to x)]$, so \eqref{e:s2sx}	follows by multiplying the above by $(1 -\alpha)$.
\end{proof}

We can now finish the proof of \cref{l:upper} in the case $\l \geq 2$.

\begin{proof}[Proof of \cref{l:upper} in the case $\l \geq 2$.]
The case $k=\l$ was proved earlier in this subsection. So it suffices to consider $k>\l\geq 2$.
By Claims~\ref{c:total0} and~\ref{c:s2} 	
we have $$s(G) \leq \alpha(1-\alpha)^m f(d) + \frac{(k-1)^{k-1}\l^\l}{(m-1)^{m-1}}(1-d) \x^m (1-\x) .$$
	By \cref{c:alphaU} we have $\alpha \leq 1/2$, and so by Lemma~\ref{l:f2} the right side is maximized as a function of $d \in [0,1]$ for some  $d \in \left[\frac{k-1}{k+\l-1},\frac{k}{k+\l}\right]$. Thus $s(G)$ is upper bounded by $\brm{OPT}(k,\l)$.
\end{proof}

\subsection{The case $\l=1$}

It remains to resolve the case $\l=1$ and so we assume $\l=1$, and thus $m=k+1$ for the duration of this subsection.
In this case crucially \cref{l:baddensity} is not applicable, and we need to work harder.

For $y \in Y$, let \begin{equation}\label{e:h}h(y) = 	
\frac{(k-1)^{k-1}}{k^k} \x^{k} \d^{+}_X(y) -    s_2(c \to y) -  \frac{\rho^{+}_X(y)}{\gamma}\bb{E}_{x \in N^{+}_X(y)}[s_X(o \to x)]. 
\end{equation}
Thus $h(y)$ is the slack in the inequality \eqref{e:s1}, so we have $h(y) \geq 0$. 
Moreover, averaging  \eqref{e:h} as in the proof of \cref{c:s2}  we can obtain the following refinement of \eqref{e:s2sx}.

\begin{claim}\label{c:s2revised} 
	\begin{equation}\label{e:s2revised}
	\frac{(k-1)^{k-1}}{k^k} (1-d) \x^m (1-\x) \geq s_2(G) +  s_X(G) + (1 -\alpha)\bb{E}_{y \in Y}[h(y)].
	\end{equation}
\end{claim}

Next we need a separate estimate of $s_X(G)$.

\begin{claim}\label{c:sx} 
	\begin{equation}\label{e:l1sx}
	s_X(G) \leq \frac{(k-1)^{k-1}}{(k+1)^{k+1}}\x^{k+2}. 
	\end{equation}
\end{claim}

\begin{proof}
	For every $x \in X$ we have
\begin{align*}
	s_X(o \to  x) &\leq \rho_X(x)\E_{v \in N^{-}_X(x)}\s{\rho^-_{X - N_X(x)}(v)(\rho^+_{X - N_X(x)}(v))^{k-1}} 
	 \\ &\leq \rho_X(x) \s{\frac{(k-1)^{k-1}}{k^k}(\x -\rho_X(x))^{k}} \\ &\leq  \frac{(k-1)^{k-1}}{k^k}\s{\frac{k^k}{(k+1)^{k+1}}\x^{k+1}} =  \frac{(k-1)^{k-1}}{(k+1)^{k+1}}\x^{k+1}. 
	\end{align*}	
Thus
$$
s_X(G) = \x \E_{x \in X}[s_X(o \to  x)]  \leq  \frac{(k-1)^{k-1}}{(k+1)^{k+1}}\x^{k+2}, 
$$	
as desired.
\end{proof}

We use \cref{c:sx} in the following lower bound on $h(y)$ which is useful  when $\d^{-}_X(y)$ is very small.
\begin{claim}\label{c:h} 
	\begin{equation}\label{e:hlower}
	h(y) \geq \frac{(k-1)^{k-1}}{k^k} \x^{k+1}  - \s{1+ \frac{(k-1)^{k-1}}{k^k} }\rho^{-}_X(y)\x^{k} -  \frac{(k-1)^{k-1}}{(k+1)^{k+1}}\frac{\x^{k+2}}{\gamma}.
	\end{equation}
\end{claim}
\begin{proof}
We obtain \eqref{e:hlower} by substituting  into \eqref{e:h} the bounds \begin{align*}&s_2(c \to y) \leq \rho^{-}_X(y)\alpha^k \qquad \mathrm{and}  &\rho^{+}_X(y)\bb{E}_{x \in N^{+}_X(y)}[s_X(o \to x)] \leq s_X(G) \leq \frac{(k-1)^{k-1}}{(k+1)^{k+1}}\x^{k+2},\end{align*}
where the last inequality is from \cref{c:sx}.
\end{proof}

The next two claims give lower bounds on the negative contribution of edges in $Y$, replacing \cref{c:sy}.

\begin{claim}\label{c:prebound2} For every $x \in X$, $y \in N^+_Y(x)$ we have
	\begin{equation}\label{e:l1s1local}
 \rho^{+}_Y(x)s_1(c \to x, o \to y) \leq (1-\alpha)^{k+1} f\s{1-\frac{\d^{-}_{Y}(x)}{1-\alpha}} - \frac{\d_Y(y)}{(1-\alpha)}S_1.
	\end{equation}
\end{claim}
\begin{proof}
	Let $$\eps_1= \frac{\mu(N_Y(y) \cap N^{-}_{Y}(x))}{1 -\alpha} \qquad \mathrm{and} \qquad \eps_2 = \frac{\d_Y(y)}{1-\alpha} - \eps_1.$$
	We have 
\begin{align*}\rho^{+}_Y(x)&s_1(c \to x, o \to y) \leq  \d^{-}_{Y - N_Y(y)}(x)\d^{+}_{Y - N_Y(y)}(x)(\rho^{+}_Y(x))^{k-1}\\ &\leq  (1-\alpha)^{k+1}\s{\frac{\d^{-}_{Y}(x)}{1-\alpha} - \eps_1}\s{1 - \frac{\d^{-}_{Y}(x)}{1-\alpha} -\eps_2} \s{1 - \frac{\d^{-}_{Y}(x)}{1-\alpha}}^{k-1} \\&  \leq (1-\alpha)^{k+1}\s{1- \eps_1-\eps_2}f\s{1-\frac{\d^{-}_{Y}(x)}{1-\alpha}},\end{align*}
where the last inequality follows from \cref{l:fractions}.
As $$ (1-\alpha)^{k+1}f\s{1-\frac{\d^{-}_{Y}(x)}{1-\alpha}} \geq (\d^{+}_{Y }(x))^k\d^{-}_{Y }(x) \geq S_1,$$
\eqref{e:l1s1local} follows.
\end{proof}

\begin{claim}\label{c:l1s1} 
	\begin{equation}\label{e:l1s1}
	s_1(G) \leq \alpha(1-\alpha)^m f(d) - S_1 \bb{E}_{y \in Y}[\d_Y(y) \d^{-}_X(y)].
	\end{equation}
\end{claim}
\begin{proof}
Averaging \eqref{e:l1s1local} over $y \in N^+_Y(x)$ we obtain 
\begin{equation}\label{e:l1s12}
s_1(c \to x) \leq (1-\alpha)^{m} f\s{1-\frac{\d^{-}_{Y}(x)}{1-\alpha}} - \frac{S_1}{1 - \alpha} \bb{E}_{y \in N^+_Y(x)}[\d_Y(y)].
\end{equation}
We now average \eqref{e:l1s12} over $x \in X$ and use concavity of $f$, as in  the proof of \cref{c:bound2l}:
\begin{align*}
s_1(G) &= \alpha \bb{E}_{x \in X}[s_1(c \to x)]   \\
& \leq \alpha (1-\alpha)^m \bb{E}_{x \in X}\left[f\s{1-\frac{\d^{-}_{Y}(x)}{1-\alpha}}\right] - S_1\frac{\alpha}{1-\alpha}\bb{E}_{x \in X}\s{ \bb{E}_{y \in N^+_Y(x)}\d_Y(y)} \\ &\leq \alpha (1-\alpha)^m f(d)- S_1 \bb{E}_{y \in Y}\s{\d_Y(y) \d^{-}_X(y)},
\end{align*}	
as desired.
\end{proof}

Finally, we need a new bound on $s_Y(G)$.

\begin{claim}\label{c:l1s2} 
	\begin{equation}\label{e:l1s2}
	s_Y(G) \leq    (1 - \x) \frac{(k-1)^{k-1}}{2k^k}\y^{k}\E_{y \in Y}\s{\rho_Y(y)}.
	\end{equation}
\end{claim}
\begin{proof}
	For every $y \in Y$ we have 
\begin{align*}	s_Y(c \to y) &\leq \s{(\rho^{+}_Y(y))^{k-1}(\rho_Y(y) - \rho^{+}_Y(y))}\rho^{+}_Y(y)  \\ 
& \leq \frac{(k-1)^{k-1}}{k^k}(\rho_Y(y))^k\rho^{+}_Y(y) \leq  \frac{(k-1)^{k-1}}{2k^k}\y^{k}\s{\rho_Y(y) + (\rho^{+}_Y(y)-\rho^-_{Y}(y))} 
	\\ &=  \frac{(k-1)^{k-1}}{2k^k}\y^{k}\rho_Y(y) + \frac{(k-1)^{k-1}}{2k^k}\y^{k}(\rho^{+}_Y(y)-\rho^-_{Y}(y)).
	\end{align*}
Then since $\E_{y \in Y}(\rho^{+}_Y(y)-\rho^-_{Y}(y))=0$, we have
\begin{align*}
s_Y(G) &= (1 - \x)\E_{y \in Y}\s{s_Y(c \to y) -  \frac{(k-1)^{k-1}}{2k^k}\y^{k}(\rho^{+}_Y(y)-\rho^-_{Y}(y))} \\
&\leq    (1 - \x) \frac{(k-1)^{k-1}}{2k^k}\y^{k}\E_{y \in Y}\s{\rho_Y(y)},
\end{align*}
as desired.
\end{proof}

\begin{claim}\label{c:l1s}	
	\begin{align}\label{e:last}
	s(G) &\leq \alpha(1-\alpha)^m f(d) + \frac{(k-1)^{k-1}}{k^{k}}(1-d) \x^m (1-\x) \notag \\ &+ (1-\alpha)\bb{E}_{y \in Y}\left( \frac{(k-1)^{k-1}}{2k^k}\y^{k}\d_Y(y) - \frac{S_2}{1 - \alpha}\d_Y(y)\d^{-}_X(y) -h(y)\right).
	\end{align}
\end{claim}
\begin{proof}
	By taking a convex combination of \eqref{e:s1i} and \eqref{e:l1s1} we obtain
	$$
	s_1(G) \leq  \alpha(1-\alpha)^m f(d) - S_2 \bb{E}_{y \in Y}[\d_Y(y) \d^{-}_X(y)] - \frac{(m-1)m}{(m+1)\binom{m}{k}}(\alpha+\beta)(1-D)^{m-2}d_0.
	$$
	Thus by \cref{c:nonedge2} we have
	\begin{equation}\label{e:l1s1final}	
	s_1(G) + s_0(G) \leq  \alpha(1-\alpha)^m f(d) - S_2 \bb{E}_{y \in Y}[\d_Y(y) \d^{-}_X(y)]. 
	\end{equation}
	Adding together \eqref{e:s2revised} and \eqref{e:l1s2} we obtain \eqref{e:last}.
\end{proof}

It now suffices to show that the contribution of every $y \in Y$ to the expectation in \cref{e:last} is non-positive, which we do in the next claim.

\begin{claim}\label{c:last} 
For every $y \in Y$ we have
	\begin{equation}\label{e:hlocal}	
 \frac{S_2}{1 - \alpha}\d_Y(y)\d^{-}_X(y) + h(y) - \frac{(k-1)^{k-1}}{2k^k}\y^{k}\d_Y(y) \geq 0.
	\end{equation}
\end{claim}
\begin{proof}
	Let $\psi = \d_Y(y)$ and $\phi = \d^{-}_X(y)$. As $h(y) \geq 0$, \eqref{e:hlocal} holds as long as
	$\frac{S_2}{1 - \alpha}\phi \geq \frac{(k-1)^{k-1}}{2k^k}\y^{k}$. 
	Thus we assume $0 \leq \phi \leq  \frac{(k-1)^{k-1}}{2k^k}\frac{(1 -\alpha)}{S_2}\y^{k}$.
	By \eqref{e:hlower} it suffices to show that 
	\begin{equation}\label{e:pp} \frac{(k-1)^{k-1}}{k^k} \x^{k+1} -  \frac{(k-1)^{k-1}}{(k+1)^{k+1}}\frac{\x^{k+2}}{\gamma} - \s{1+ \frac{(k-1)^{k-1}}{k^k} }\x^{k}\phi +  \frac{S_2}{1 - \alpha}\phi\psi - \frac{(k-1)^{k-1}}{2k^k}\y^{k}\psi \geq 0.\end{equation}
	The left side of this inequality is a linear function of $\phi$, so it suffices to verify it for $\phi=0$ and $\phi =\frac{(k-1)^{k-1}}{2k^k}\frac{(1 -\alpha)}{S_2}\y^{k}$.

	For $\phi=0$ we may additionally assume $\psi= \beta$ as it appears with the negative coefficient. Thus in this case \eqref{e:pp} reduces to
		$$ \frac{(k-1)^{k-1}}{k^k} \x^{k+1} -  \frac{(k-1)^{k-1}}{(k+1)^{k+1}}\frac{\x^{k+2}}{\gamma} - \frac{(k-1)^{k-1}}{2k^k}\y^{k+1} \geq 0,$$
		which holds by \eqref{e:final}.
	
		Finally, for  $\phi =\frac{(k-1)^{k-1}}{2k^k}\frac{(1 -\alpha)}{S_2}\y^{k}$,  \eqref{e:pp} reduces to \begin{equation}\label{e:pp2} \frac{(k-1)^{k-1}}{k^k} \x^{k+1} -  \frac{(k-1)^{k-1}}{(k+1)^{k+1}}\frac{\x^{k+2}}{\gamma} - \s{1+ \frac{(k-1)^{k-1}}{k^k} }\x^{k}\frac{(k-1)^{k-1}}{2k^k}\y^{k}\frac{(1 -\alpha)}{S_2} \geq 0.\end{equation}
		By Claims \ref{c:gamma} and \ref{c:S2}, we have $\frac{\alpha}{\gamma} \leq 2$ and
		$S_2 \geq \frac{1 - \alpha}{\alpha}\beta^{k}.$ Dividing \eqref{e:pp2} by $\frac{(k-1)^{k-1}}{k^k}\alpha^{k+1}$ and substituting the above bounds we obtain
		$$1 - \frac{2k^k}{(k+1)^{k+1}} - \frac{1}{2}\s{1+ \frac{(k-1)^{k-1}}{k^k}} \geq 0,$$
		which clearly holds as $k \geq 5$.
\end{proof}

With all the ingredients in place we can easily finish the proof of \cref{l:upper} in the last remaining case.

\begin{proof}[Proof of \cref{l:upper} in the case $\l = 1$.]
	By Claims~\ref{c:l1s} and~\ref{c:last} we have 	
 $$s(G) \leq \alpha(1-\alpha)^{m} f(d) + \frac{(k-1)^{k-1}\l^\l}{(m-1)^{m-1}}(1-d) \x^m (1-\x) \leq \brm{OPT}(k,1),$$
	where the second inequality follows from Lemma~\ref{l:f2} (just as in the case $k > \l \geq 2$).
\end{proof}

\section{Concluding remarks}\label{s:remarks}

\subsubsection*{Stability}

The inducibility problem solved in \cref{Thm:main} is \emph{stable} in the following sense: For $k+\l \geq 6$ every sufficiently large graph with the induced density of $S_{k,\l}$ sufficiently close to the maximum has the structure which is close to the optimal one described in the proof of \cref{l:lower}. We make this statement precise for $k>\l$ in the next theorem.\footnote{In the case $k=\l$ the analogous result also holds.}  It is obtained by careful, yet straightforward examination of the inequalities used in the proof of \cref{Thm:main}. We omit the details.

\begin{thm}\label{t:stability} For all integers $k>\l \geq 1$ such that  $k+\l \geq 6$ and every $\eps > 0$, there exist  $\delta, n_0 > 0$ satisfying the following. 
	
	Let $(\alpha,d) \in [0,\frac12]\times [0,\frac{k}{k+\ell}]$ maximize the expression in the statement of \cref{Thm:main}, and let $G$ be a digraph such that $n=|V(G)|\geq n_0$, and $i(\SK,G) \geq (1 - \delta)i(\SK).$  
	Then there exists a partition $(X,Y_1,Y_2)$ of $V(G)$ such that \begin{enumerate} \item $|\mu(X) - \alpha| \leq \eps$, \item $|\mu(Y_1) -\frac{k+\l-1}{k-1}(1-d)(1-\alpha)| \leq \eps $ and $|\mu(Y_2) -\s{1-\frac{k+\l-1}{k-1}(1-d)}(1-\alpha)| \leq \eps$, 
		\item $|E(G[X])| + |E(G[Y_1 \cup Y_2])| \leq \eps n^2$,
		\item at most $\eps n$ vertices $x \in X$ fail to satisfy $$|\d^+_{Y_1 \cup Y_2}(x) - d(1-\x)| \leq \eps \qquad \mbox{and} \qquad |\d^-_{Y_1 \cup Y_2}(x) - (1-d)(1-\x)| \leq \eps,$$
		\item at most $\eps n$ vertices $y \in Y_1$ fail to satisfy  $|\d^+_X(x) - \frac{l}{k+\l-1}\x| \leq \eps$ and $|\d^-_X(x) - \frac{k-1}{k+\l-1}\x| \leq \eps$, and 
		\item finally, at most $\eps n$ vertices $y \in Y_2$ fail to satisfy  $\d^-_X(x) \geq \x -\eps$.
		 \end{enumerate}
\end{thm} 

It is likely that the methods used in the proof of  \cref{Thm:main} can also be used to obtain an exact version of \cref{t:stability}. More precisely, we believe that  if $G$ is an $n$-vertex  digraph, which satisfies $i(\SK,G) = i(\SK,n)$ for $n$ sufficiently large as a function of $\eps,$ then there exists a partition  $(X,Y_1,Y_2)$ such that the conditions 1 and 2 of \cref{t:stability} still hold, conditions 4, 5 and 6 hold for all vertices of $X,~Y_1$ and $Y_2$, respectively, and condition 3 is replaced by the following stronger property   
\begin{itemize} \item[3$'$.]  $G[X]$ and $G[Y_1 \cup Y_2]$ are edgeless, and every vertex of $X$ is adjacent to every vertex of $Y_1 \cup Y_2$. 
\end{itemize}
However, unlike \cref{t:stability}, the above result does not directly follow from the bounds we established, and we leave its validity open.

\subsubsection*{Approximating the optimum}

\cref{Thm:main} expresses inducibility of oriented stars in terms of a solution to a polynomial optimization problem. This is unavoidable, as in general, the resulting optimization problem has no closed form solution. However, it is possible to approximate this solution, and thus the inducibility with great precision.

For example, for $k>\l$  considering Taylor series of 
$$F(\alpha,d)= \alpha(1-\alpha)^{k+\l}d^k(1-d)^\ell+\frac{(k-1)^{k-1}\l^\l}{(k+\ell-1)^{k+\l-1}}(1-\alpha)\alpha^{k+\l}(1-d).$$
at a point $(\frac{1}{k+\l+1},\frac{k}{k+\l})$, we obtain that $F$ is maximized on $[0,\frac12]\times [0,\frac{k}{k+\ell}]$	when \begin{align*}\alpha&=\frac{1}{k+\l+1}\s{1+\frac{\l}{k(k+\l)^{k+\l-2}}+o_k\s{\frac{1}{(k+\l)^{k+\l-1}}}}, \\
d &= \frac{k}{k+\l}\s{1-\frac{\l}{k(k+\l)^{k+\l}}+o_k\s{\frac{1}{(k+\l)^{k+\l+1}}}},
\end{align*}
and 
\begin{align*}\max_{(\alpha, d)\in [0,\frac12]\times [0,\frac{k}{k+\ell}]} F(\alpha,d) & = \frac{k^k\l^\l}{(k+\l+1)^{k+\l+1}} + \frac{(k-1)^{k-1}\l^{\l+1}(k+\l)}{(k+\l+1)^{k+\l+1}(k+\ell-1)^{k+\l-1}} \times \\ &  \s{1+\frac{\l}{2k(k+\l)^{k+\l-3}}+o_k\s{\frac{1}{(k+\l)^{k+\l-2}}}}.\end{align*}
Already for $k=4,\l=2$ the resulting approximation of the solution to our maximization problem is correct up to eight significant digits. 

\subsubsection*{Flag algebras}

\cref{Thm:main} might extend to the cases $(k,\l) \in \{(2,1),(3,1),(4,1)\}$, although we were unable to resolve these cases using our techniques. These  cases, however, might be amenable to analysis using flag algebras.
For example, we are able to solve the case $(k,\l)=(2,1)$. But since it is standard application of flag algebras we omit details here. 
 (see Appendix~\ref{sec:fm} for numerical computations using Flagmatic \cite{JS,EV} for $(k,\l)=(2,1)$).

\bibliographystyle{abbrv}
\bibliography{snorin}

\appendices

\section{Proofs of the technical claims}\label{sec:claim}

\begin{claim}\label{A:S}  	
	\begin{equation}\label{eq:S}
	S \geq \frac{k^k\l^\l}{(m+1)^m}.
	\end{equation}
\end{claim}
\begin{proof}
For any $(\alpha,d)\in [0,\frac12]\times [0,\frac{k}{k+\ell}]$ we have  $$ S \geq (m+1)\brm{OPT}(k,\l) > (m+1)\alpha(1-\alpha)^{m}d^k(1-d)^\ell.$$ Substituting $\alpha = \frac{1}{m+1}, d = \frac{k}{m}$, we obtain the claimed bound.
\end{proof}

The next claim  immediately follows by substituting \eqref{eq:S} into the definitions of $\lambda_0$ and $\lambda_1$. 

\begin{claim}\label{A:lambdas}  	
	\begin{equation}\label{eq:lambdas}
	\lambda_0 \leq \left( \frac{m+1}{m}\right)^m S \qquad \mathrm{and} \qquad \lambda_1 \leq  \frac{(m+1)^m}{(m-1)^{m-1}}S.
	\end{equation}
%	In particular, for $\l \geq 2$
%	\begin{equation}\label{eq:lambda1}\lambda_1 < \frac{(m+1)^m}{2(m-1)^{m-1}}S.
%	\end{equation}	
\end{claim}

\begin{claim}\label{A:xdegree}  
 $$D \geq \left(1-\frac{1}{m^3}\right) \frac{m}{m+1}.$$
\end{claim} 

\begin{proof} Let $x \in X$ be such that  $\rho(x) = D$. 
	
	Substituting the upper bounds on $\lambda_0$ and $\lambda_1$ from \cref{A:lambdas} into \eqref{eq:sx} we obtain	\begin{align}\label{eq:xdegree2}
	\frac{D^m}{m^{m}}  + \frac{D(1 - D)^{m-1} }{(m-1)^{m-1}} \geq \frac{1}{(m+1)^{m}}.
	\end{align}	As the left side of \eqref{eq:xdegree2} is convex for $D \geq 2/m$, it suffices to verify that \eqref{eq:xdegree2} does not hold for $D =  (1-\frac{1}{m^3}) \frac{m}{m+1}$ and $D = 1/2$.

	We consider $D =  (1-\frac{1}{m^3}) \frac{m}{m+1}$ first. After substituting the value of $D$ into  \eqref{eq:xdegree2}, multiplying by $(m+1)^m$ and rearranging, it remains to verify that
	$$
	\s{1 -\frac{1}{m^3}}^m + \frac{(1-\frac{1}{m^3})m}{(m-1)^{m-1}} \left(1  + \frac{1}{m^3}\right)^{m-1} < 1.
	$$
	As $$\left(1  + \frac{1}{m^3}\right)^{m-1} < 1+\frac{1}{m^2},\qquad \s{1 -\frac{1}{m^3}}^m \leq 1-\frac{1}{m^2}+\frac{1}{m^4} \qquad \mbox{and} \qquad 1-\frac{1}{m^3}\leq 1, $$ the above is implied by
	$$
    \frac{1}{m^4} + \frac{m^2+1}{m(m-1)^{m-1}} \leq \frac{1}{m^2},
	$$
	which clearly holds for $m \geq 6$.
	\medskip
	
	Next let $D = 1/2$. After multiplying \eqref{eq:xdegree2} by $2^m(m+1)^m$, we need to show
	$$\left(1+\frac{1}{m}\right)^m + (m+1)\left(1+\frac{2}{m-1}\right)^m  < 2^m.$$  
	It is easy to see that it suffices to verify this inequality for $m=6$, which is straightforward.
\end{proof}

Next an easy consequence of the above claim.

\begin{claim}\label{A:Dpower} For  $0 \leq p \leq m-1$,
	$$(1-D)^p < \frac{1}{m(m+1)^{p-1}}.$$	
\end{claim} 
\begin{proof} By \cref{A:xdegree} we have
	\begin{align*}
	(1-D)^{p} &\leq  \frac{1}{(m+1)^p}\s{1+\frac{1}{m^2}}^p
	\leq \frac{\exp\s{\frac{p}{m^2}}}{(m+1)^p} \leq \frac{\exp\s{\frac{1}{m+1}}}{(m+1)^p} \leq \frac{1}{m(m+1)^{p-1}}.
	\end{align*}
\end{proof}	

We are now ready to prove \cref{c:alphaU}, that is to show $\x \leq 1/m$.

\begin{proof}[Proof of \cref{c:alphaU}]
	By \eqref{e:alpha},  
$$
	\lambda_0(1-\alpha)^m + \lambda_1(1+\alpha)(1-D)^{m-1}\geq S.
$$
	Substituting the upper bounds on $\lambda_0$ and $\lambda_1$ from \cref{A:lambdas}, and on $(1-D)^{m-1}$ from \cref{A:Dpower}, we obtain
		\begin{equation}\label{e:alphaU1}\fs{m+1}{m}^m(1-\alpha)^m + \frac{(m+1)^2}{m(m-1)^{m-1}}(1 +\alpha) \geq 1.	\end{equation}
 As the left side of \eqref{e:alphaU1} clearly decreases with $\alpha$ for $\alpha \leq 1/m$, it suffices to show that  \eqref{e:alphaU1} does not hold for $\alpha = 1/m$, i.e.\@ that 
	$$\s{1-\frac{1}{m^2}}^m + \frac{(m+1)^3}{m^2(m-1)^{m-1}} < 1.$$
	Using the bound $\s{1-\frac{1}{m^2}}^m \leq e^{-1/m} < 1-\frac{1}{m+1}$, we reduce the above to
	$$(m+1)^4 < m^2(m-1)^{m-1},$$
	which comfortably holds for $m \geq 6$.
\end{proof}

\begin{claim}\label{A:beta} 
	$$\beta <\s{1+\frac{2}{m(m-1)^3}}\frac{1}{m+1}.$$
\end{claim}

\begin{proof}
	Using the upper bounds on $\lambda_0$ and $\lambda_1$ from \cref{A:lambdas}, as well as upper bounds on powers of $1-D$ and on $\alpha$ from \cref{A:Dpower} and \cref{c:alphaU}, respectively, we see that \cref{c:B} implies  that  either
	
		\begin{align} 
		&\frac{(m+1)^m}{(m-1)^{m-1}} \alpha (1-\alpha - \beta)^{m-1} + \frac{(m+1)^2}{m^2(m-1)^{m-1}} +\s{\frac{(m+1)^m}{(m-1)^{m-1}} + \frac{(m+1)^m}{m^m}} \beta^{m}  \geq 1, \qquad \mathrm{or}  \label{e:beta3} \\
		&\s{\frac{(m+1)^m}{(m-1)^{m-1}} + \frac{(m+1)^m}{m^m}} \beta^{m}  +  \frac{(m+1)^3}{m k^k\l^\l}\frac{(m-1)}{\binom{m}{k}}\s{\frac{1}{m} + 
		\beta} \geq 1. \label{e:beta2} 
		\end{align}	
		
	As $\beta \leq 1/2$, we have
	\begin{align*}
	\frac{(m+1)^3}{m k^k\l^\l}\frac{(m-1)}{\binom{m}{k}}\s{\frac{1}{m} + 
		\beta} &\leq \frac{k!}{k^k}\frac{(m+1)^3(m+2)(m-1)}{2 m^2 \cdot m!} \leq \frac{2}{9}\frac{(m+1)^3}{2m!} \leq \frac{343}{6480} \leq 0.1
	\end{align*}
	and
	\begin{equation}\label{e:betam}
	\s{\frac{(m+1)^m}{(m-1)^{m-1}} + \frac{(m+1)^m}{m^m}} \beta^{m}  \leq \frac{e^2(m+1)+e}{2^m} \leq \frac{55}{64} \leq 0.7.
	\end{equation}

Thus \eqref{e:beta2} does not hold, and so \eqref{e:beta3} holds. 
	By \cref{l:xayb} we have
	$$ \alpha(1-\alpha - \beta)^{m-1} \leq \frac{(m-1)^{m-1}}{m^m}(1- \beta)^m.$$
	and so
	\begin{equation}\label{e:Abeta}
\left(\frac{m+1}{m}\right)^m(1-\beta)^m+  \frac{(m+1)^2}{m^2(m-1)^{m-1}}	+ \left(\left(\frac{m+1}{m}\right)^m+ \frac{(m+1)^m}{(m-1)^{m-1}}\right) \y^{m}  \geq 1.
	\end{equation}
	As the left side of \eqref{e:Abeta}  is convex, and it is easy to show using \eqref{e:betam} that \eqref{e:Abeta} does not hold for $\beta=1/2$, it suffices to show that  \eqref{e:Abeta}  also does not hold for $\beta = (1+m\eps)/(m+1)$, where $\eps= \frac{2}{m^2(m-1)^3}$. As $1 - \beta = (1-\eps)\frac{m}{m+1}$,  we rewrite  \eqref{e:betam} as
\begin{equation}\label{e:Abeta2}(1-\eps)^m + \frac{(m+1)^2}{m^2(m-1)^{m-1}} +\left(
\frac{1}{m^m}+ \frac{1}{(m-1)^{m-1}}\right)(1+m\eps)^m \geq 1.	\end{equation}
	As $\eps \leq 1/m^3$ we have $(1+m\eps)^m < 1+(m+1)m \eps$ and $(1 -\eps)^m \leq 1 - (m-1)\eps$, so
	\eqref{e:Abeta2} implies 
	$$\s{m-1 - \frac{m(m+1)}{m^m} - \frac{m(m+1)}{(m-1)^{m-1}}}\eps  \leq \frac{(m+1)^2}{m^2(m-1)^{m-1}} + \frac{1}{m^m}+ \frac{1}{(m-1)^{m-1}}, $$
	which in turn implies
	$$
	(m-2)\eps < \frac{2(m+1)^2}{m^2(m-1)^{m-1}}< \frac{2(m-2)(m-1)^2}{m^2(m-1)^{m-1}} \leq \frac{2(m-2)}{(m-1)^3} ,
	$$
	contradicting our choice of $\eps$.
\end{proof}

\begin{claim}\label{A:alphaL} 
		$$\left(1 - \frac{2}{(m-1)^{2}} \right)\frac{1}{m+1}\leq \alpha.$$
\end{claim}
\begin{proof}
	By \cref{A:beta} we have $$\beta^m \leq \s{1+\frac{1}{m(m+1)}}^m\frac{1}{(m+1)^m} \leq \frac{\exp\fs{1}{m+1}}{(m+1)^m} \leq \frac{1}{m(m+1)^{m-1}}.$$ Substituting this bound, as well as the bounds on $\lambda_0,\lambda_1$ and $(1-D)^{m-1}$ from \cref{A:lambdas} and  \cref{A:Dpower} into \eqref{e:AL}, we obtain
	\begin{equation}\label{e:AAL}
	\frac{(m+1)^{m+1}}{m^m} \x (1-\x)^{m} + \frac{(m+1)^{3}}{m(m-1)^{m-1}}\x + 	\frac{(m+1)^{2}}{m^{m+1}}(1-\x) \geq 1.
	\end{equation}
	As the left part of \eqref{e:AAL} is increasing for $\x \leq \frac{1}{m+1}$, it suffices to show that \eqref{e:AAL} does not hold for $\x = \left(1 - \frac{2}{(m-1)^{2}} \right)\frac{1}{m+1}$. Let $\eps =  \frac{2}{(m-1)^{2}}$. We upper bound different terms in  \eqref{e:AAL} as follows:
	\begin{align*}
	\frac{(m+1)^{m+1}}{m^m}& \x (1-\x)^{m} = (1-\eps)\s{1+\frac{\eps}{m}}^m \\  &\leq (1-\eps)e^{\eps} \leq (1-\eps)\s{1 + \eps +\frac{\eps^2}{2}+\frac{\eps^3}{2}} \leq  1 - \frac{\eps^2}{2} = 1 - \frac{2}{(m-1)^4} ,
	\end{align*} where the third inequality holds as $\eps \leq 1$, and
	\begin{align*}
	\frac{(m+1)^{3}}{m(m-1)^{m-1}}\x + 	\frac{(m+1)^{2}}{m^{m+1}}(1-\x) \leq \frac{(m+1)^{2}}{m(m-1)^{m-1}}+\frac{m+1}{m^{m}}.
		\end{align*}
	Plugging these bound in \eqref{e:AAL} yields
	\begin{align*}
	 0 \leq \frac{(m+1)^{2}}{m(m-1)^{m-1}} + \frac{(m+1)}{m^{m}} - \frac{2}{(m-1)^4}  \leq \frac{1}{(m-1)^{m-2}} \s{\frac{(m+1)^{2}}{m(m-1)} + \frac{(m+1)(m-1)^{m-2}}{m^{m}} - 2},
	\end{align*}
 where the last inequality holds, as $\frac{(m+1)^{2}}{m(m-1)} + \frac{(m+1)(m-1)^{m-2}}{m^{m}}$ decreases with $m$ and evaluates to about $1.727$ for $m=6$. Thus, as desired, \eqref{e:AAL} does not hold for our choice of $\eps$.
\end{proof}

\begin{claim}\label{A:S1} 
	$$S_1 \geq \frac{44}{45}S.$$
\end{claim}
\begin{proof} By \eqref{e:alpha}, plugging the previously obtained bounds we have
	\begin{align*}
\frac{	S_1}{S} &\geq 1 -  \frac{\lambda_1}{S}(1+\x)(1-D)^{m-1} \\&\geq 1 - \frac{(m+1)^{m+1}}{(m-1)^{m-1}}\frac{m+1}{m}\frac{1}{m(m+1)^{m-2}} = 1 - \frac{(m+1)^4}{m^2(m-1)^{m-1}} \geq \frac{44}{45},
	\end{align*}
	where the last inequality holds for $m=6$, and so for all $m \geq 6$.
\end{proof}

\begin{proof}[Proof of \cref{c:gamma}] We have
	\begin{align*}
	\gamma^\l(1-\x - \gamma)^k \leq \frac{\gamma}{\x + \gamma} \s{(\gamma+\x)^\l(1-\x - \gamma)^k} \leq \frac{k^k\l^\l}{m^m}\frac{\gamma}{\x + \gamma}.
	\end{align*}
 On the other hand, by \eqref{e:alpha} and Claims \ref{A:S} and \ref{A:S1},
 \begin{align*}
 \gamma^\l(1-\x - \gamma)^k \geq S_1 \geq   \frac{44}{45}S \geq \frac{44}{45}\frac{k^k\l^\l}{(m+1)^{m}} \geq \s{\frac{44}{45}\fs{m}{m+1}^m} \frac{k^k\l^\l}{m^{m}} \geq \frac{1}{3}\frac{k^k\l^\l}{m^{m}}.
 \end{align*}
 It follows that $3 \gamma \geq \x + \gamma$, as desired.
\end{proof}

Having estimated all the quantities present in the statements of Claims~\ref{c:S2} and~\ref{c:final}, we are now ready to prove them.
\begin{proof}[Proof of \cref{c:S2}] 
By Claims~\ref{c:alphaU}, \ref{A:S}, \ref{A:Dpower} and \ref{A:beta}, we have
	\begin{align*}
	\frac{1}{S} &\frac{(m-1)}{(m+1)\binom{m}{k}}(\alpha+\beta)(1-\alpha)(1-D)^{m-2} \\ &\leq \frac{k!}{k^k}\frac{(m-1)(m+1)^{m-1}}{m!}\cdot \frac{2}{m}\cdot \frac{1}{m(m+1)^{m-3}} \\ &\leq 
	\s{ \frac{2}{9}\cdot \frac{2(m+1)^2}{m^2}} \frac{m-1}{m!} \leq \frac{1}{144}.
	\end{align*}
	It follows from \cref{A:S1} that
	\begin{align*}
	\frac{S_2}{S} = \frac{S_1}{S} - \frac{1}{S} &\frac{(m-1)}{(m+1)\binom{m}{k}}(\alpha+\beta)(1-\alpha)(1-D)^{m-2}  \geq \frac{44}{45}-\frac{1}{144} \geq \frac{32}{33}.
	\end{align*}
	On the other hand.	
	\begin{align*}
	\frac{\lambda_1}{S} \frac{1 - \alpha}{\alpha}\beta^{m-1} \leq \frac{(m+1)^{m}}{(m-1)^{m-1}} \frac{(m+1)}{1 -\frac{2}{(m-1)^2}}\frac{1}{m(m+1)^{m-2}} \leq \frac{(m+1)^3}{(m-1)^{m}} \leq 0.03,
	\end{align*}	
	where the last inequality holds for $m=6$ and so for all $m$. Therefore \cref{c:S2} holds.
\end{proof}

\begin{proof}[Proof of \cref{c:final}]
The statement of \cref{c:final} can be rewritten as	
		\begin{equation}\label{e:final2}
		\s{1 - \frac{k^k}{(k+1)^{k+1}}\frac{\x}{\gamma}}\fs{\x}{\y}^{m} \geq  \frac{1}{2}.
		\end{equation}
As $k \geq 5$, and $\x/\gamma \leq 2$ by \cref{c:gamma}, we have
 	$$1 - \frac{k^k}{(k+1)^{k+1}}\frac{\x}{\gamma} \geq \frac{20203}{23328}.$$
 Meanwhile, by Claims~\ref{A:beta} and~\ref{A:alphaL}, we have
 	\begin{align*} 	
 	\frac{\x}{\y} \geq \frac{1 - \frac{2}{(m-1)^{2}}}{1+\frac{2}{m(m-1)^3}} \geq 1 - \frac{2}{(m-1)^{2}}-\frac{2}{m(m-1)^3} \geq 1 -\frac{2}{m(m-2)}.
 	\end{align*}
Therefore 
$$
   	\fs{\x}{\y}^{m} \geq	\s{1-\frac{2}{m(m-2)}}^m \geq \exp\s{-\frac{m}{m(m-2)/2-1}} \geq e^{-\frac{6}{11}}.
$$
   	Combining these estimates we have
   	$$\s{1 - \frac{k^k}{(k+1)^{k+1}}\frac{\x}{\gamma}}\fs{\x}{\y}^{m} \geq  \frac{20203}{23328}e^{-\frac{6}{11}} \geq \frac{1}{2}, $$
   	and so \eqref{e:final2} holds.	
\end{proof}

\section{Flagmatic}\label{sec:fm}
Here we use Flagmatic maintained by Slia\v can~\cite{JS}. First install Sage~\cite{sage}, then download Flagmatic from \url{https://github.com/jsliacan/flagmatic}, which also contains directions on how to install and run it. The following code will give a numerical result for $S_{2,1}$.

\begin{verbatim}
from flagmatic.all import *
P = OrientedGraphProblem(5,density="4:121341")
P.solve_sdp(solver="csdp")
\end{verbatim}

The output would be 0.2025, which is the same as the conclusion from Theorem~\ref{Thm:main} for $k=2$ and $\l=1$, where the maximum is achieved by $(\x,d) = (3/10,9/14)$.

\end{document}